\documentclass[10pt]{amsart}
\usepackage{amssymb, amsmath, amsthm,amsrefs}
\usepackage{latexsym}
\usepackage{color}
\usepackage{exscale}
\usepackage{fullpage}
\usepackage{bm, bbm, mathrsfs}

\newtheorem{theorem}{Theorem}[section]

\newtheorem{proposition}[theorem]{Proposition}

\newtheorem{problem}{Problem}
\newtheorem{definition}[theorem]{Definition}

\newtheorem{example}[theorem]{Example}
\newtheorem{observation}[theorem]{Observation}
\newtheorem{remark}[theorem]{Remark}

\def\s0{{ s_0}}

\def\ts0{{\tilde s_0}}

\def\tor{{\bf T}}
\def\cur{{\bf R}}

\def\eq#1{(\ref{#1})}
\def\nn{\nonumber}

\def\({\left(\begin{array}{cccccc}}
\def\){\end{array}\right)}

\def\bes{\begin{eqnarray}}
\def\ees{\end{eqnarray}}

\newcommand{\del}{\partial}
\newcommand{\pd}[2]{\frac{\partial{#1}}{\partial{#2} }}
\newcommand{\beq}{\begin{equation}}
\newcommand{\eeq}{\end{equation}}
\newcommand{\bea}{\begin{eqnarray}}
\newcommand{\eea}{\end{eqnarray}}
\newcommand{\beann}{\begin{eqnarray*}}
\newcommand{\eeann}{\end{eqnarray*}}
\newcommand{\eps}{\ensuremath{\epsilon}}

\newcommand{\lam}{\ensuremath{\lambda}}
\newcommand{\bet}{\ensuremath{\beta}}
\newcommand{\Beta}{\ensuremath{\mathcal B}}

\newcommand{\RR}{\mathbb{R}}
\newcommand{\R}{\ensuremath{\mathfrak{R}}}
\newcommand{\ti}{\tilde}

\newcommand\br[1]{\overline{{#1}}}
\newcommand{\pf}{\begin{proof}}
\newcommand{\foorp}{\end{proof}}
\newcommand{\nquad}{\negthickspace\negthickspace
\negthickspace\negthickspace}
\newcommand{\nqquad}{\nquad\nquad}

\DeclareMathOperator{\rank}{rank}

\DeclareMathOperator{\diag}{diag}

\numberwithin{equation}{section}
\begin{document}
\title{Extensions for Systems of Conservation Laws}
\author{Helge Kristian Jenssen }
\address{ H.~K.~Jenssen, Department of Mathematics, Penn State University, University Park, 
State College, PA 16802, USA ({\tt jenssen@math.psu.edu}).}
\thanks{H.~K.~Jenssen was partially supported by NSF grants DMS-0539549 (CAREER) and DMS-1009002.}
\author{Irina A. Kogan}
\address{Irina A. Kogan, Department of Mathematics, North Carolina State University, 
Raleigh, NC, 27695, USA  ({\tt iakogan@ncsu.edu}).}
\thanks{ I.~A.~Kogan was partially supported by NSF grant CCF-0728801 and NSA grant H98230-11-1-0129}
\date{\today}

\begin{abstract}
	Entropies (convex extensions) play a central role in the theory of hyperbolic conservation laws
	by providing intrinsic selection criteria for weak solutions and local well-posedness 
	for the Cauchy problem. While many systems occurring in physical models are equipped with 
	extensions, it is well-known that existence of a non-trivial (i.e.\ non-linear) extension requires
	the solution of an over-determined system of equations.
	On the other hand, so-called rich systems are equipped with large sets of entropies. 
	Beyond these general facts little seems to be known about ``how many" extensions a particular 
	system of conservation laws has. 
	
	
	For a given hyperbolic system $u_t+f(u)_x=0$, a standard approach is to analyze directly 
	the second order PDE system for the extensions. 
	Instead we find it advantageous to consider the equations satisfied by the 
	lengths $\beta^i$ of the right eigenvectors $r_i$ of $Df$, as measured with respect to the inner 
	product defined by an extension. For a given eigen-frame $\{r_i\}$ the extensions are determined 
	uniquely, up to trivial affine parts, by these lengths. 
	
	This geometric formulation provides a natural and systematic approach to existence of extensions. 
	By considering the eigen-fields $r_i$ as prescribed our results automatically apply to all systems with the same 
	eigen-frame. As a computational benefit we note that the equations for the lengths $\beta^i$ form a 
	first order algebraic-differential system (the $\beta$-system) to which standard integrability theorems 
	can be applied. The size of the set of extensions follows by determining the 
	number of free constants and functions present in the general solution to the $\beta$-system.
	We provide a complete breakdown of the various possibilities for $3\times 3$-systems, as well as 
	for rich frames in any dimension provided the $\beta$-system has trivial algebraic part. The latter
	case covers $2\times 2$-systems, strictly hyperbolic rich systems of any size, and 
	any rich system with an orthogonal eigen-frame.
	
	Our analysis is relevant whenever there exists a non-trivial conservative system whose eigen-frame 
	coincides with the given frame. This issue was analyzed by the authors in \cite{jk1}, where the problem 
	was formulated in terms of another algebraic-differential system, the ``$\lambda$-system," whose solutions  
	provide the characteristic speeds (eigenvalues) of the resulting conservative systems.
	We investigate the relationships between the $\lambda$- and $\beta$-systems and recover 
	standard results for symmetric systems (orthogonal frames).  
	It turns out that despite close structural connections between the $\lambda$- and the $\beta$-system, 
	there is no general relationship between the sizes of their solution sets.
	
	We provide a list of examples that illustrate our results.
\end{abstract} 

\maketitle

\tableofcontents

\section{Introduction, background, and discussion} \label{intro}
\subsection{Notation and conventions} Unless otherwise stated the following will be in force:
\begin{itemize}
\item $u=(u^1,\dots,u^n)$  denotes a fixed coordinate system on an open domain 
	$\Omega\subset\RR^n$. The domain $\Omega$ is assumed to be smoothly contractible to a point. 
	
\item	We denote the $(i,j)$-entry (i.e., the element in the $i$th row and the $j$th column) 
	of a matrix $A$ by $A^i_j$. Superscript $\,{}^T$ denotes transpose. 
	\item All vectors and vector functions are assumed to be 
	column vectors, except gradients of scalar maps $\phi:\RR^n\to\RR$ which are row vectors: 
	$\nabla\phi=\big(\frac{\del \phi}{\del u^1},\dots,\frac{\del \phi}{\del u^n}\big)\in\RR^{1\times n}$. 
	
	\item The Hessian of $\phi$ is $D^2\phi=\big(\frac{\del^2 \phi}{\del u^i\del u^j}\big)\in\RR^{n\times n}.$
	For a map $f:\Omega\to\RR^n$ the Jacobian matrix of $f$ is denoted by 
	$Df=\big(\frac{\del f^i}{\del u^j}\big)\in\RR^{n\times n}$. An $n\times n$-matrix $A(u)$ is a $u$-Hessian 
	($u$-Jacobian) if there is a map 
	$\phi:\mathbb R^n\to\mathbb R$ ($f:\mathbb R^n\to\mathbb R^n$) such that 
	$A(u)=D^2\phi(u)$ ($A(u)=Df(u)$). For emphasis we sometimes use a subscript to indicate the 
	coordinates in which Hessian or Jacobian are computed:
	$A=D^2_u\phi$  ($A=D_uf$). 
	\item 
	An ``inner product" is a  symmetric, but not necessarily positive definite, $2$-tensor.
	We also refer to symmetric $n\times n$-matrices as inner products on $\RR^n$.
	\item Summation convention is {\em not} used. 
	\item We write $\eps(i,j,k)=1$ to mean $i\neq j\neq k\neq i$. 
	\item For convenience, by ``smooth" we mean $C^k$-smooth for some sufficiently large $k\geq 2$.
\end{itemize}

\subsection{Hyperbolic conservation laws and entropies}
Consider a system of $n$ conservation laws in one spatial dimension
\beq\label{cons} 
	u_t+f(u)_x = 0\,,\qquad t,\, x\in \RR\,,
\eeq
where the unknown $u=u(t,x)=(u^1(t,x),\dots,u^n(t,x))^T\in\RR^n$ is the column  vector of 
conserved quantities. The map $f:\Omega\subset\RR^n\to\RR^n$ is referred to as the 
flux function. We consider the case of {\em hyperbolic} systems for which the Jacobian
$Df(u)$ is diagonalizable over $\RR$ and with a basis of eigenvectors at each $u\in\Omega$.
Many 1-dimensional systems \eq{cons} are derived 
from multi-dimensional models in continuum mechanics by assuming that $u(t,\cdot)$ 
varies along only one fixed spatial direction, \cite{daf}. The example par excellence is 
the 1-d compressible Euler equations that model uni-directional inviscid gas flow, \cites{daf, gr, sm}

It is well-known that solutions of the initial value problem for \eq{cons}
generically develop discontinuities in finite time \cites{daf}. 
In the context of the Euler system this provides a model for gas-dynamical shocks, i.e.\ narrow 
transition regions where the flow suffers steep gradients. To progress the solution 
beyond shock formation it is necessary to admit weak (distributional) solutions. However,
extending the solution space to include discontinuous functions introduces non-uniqueness:
solutions are not unique within the full class of all weak solutions. A central issue
is to regain uniqueness by imposing appropriate selection criteria. 

Motivated by physical models it makes good sense to consider
\eq{cons} as an idealization which discards higher order dissipative terms.
One approach is to admit a solution $u(t,x)$ of \eq{cons} only if it can 
be realized as a vanishing viscosity limit of solutions $u^\varepsilon(t,x)$ to
\beq\label{visc}
	u^\varepsilon_t+f(u^\varepsilon)_x=\varepsilon u_{xx}^\varepsilon\,,
	\qquad\text{as $\varepsilon\downarrow 0$.}
\eeq
However, the construction and convergence of solutions to \eq{visc} is a notoriously hard 
problem in itself \cite{bb}.
A more intrinsic approach, formalized by Kru{\v z}kov \cite{kru} and Lax \cite{lax1},
is to impose so-called {\em entropy inequalities}. These are derived from
\eq{visc} by considering the equation satisfied by $\eta(u^\varepsilon)$, where $\eta$
is a given scalar field intended to generalize the physical entropy in gas-dynamics. 
Since the classical entropy is a convex function of the state variables
(local thermodynamic equilibrium, \cite{gr}) we insist that the Hessian matrix $D^2\eta$ is positive semidefinite. 
Assuming for now that there is an associated entropy flux $q$, i.e.\ $\nabla q=\nabla \eta Df$, we obtain
from \eq {visc} that
\[\eta(u^\varepsilon)_t+q(u^\varepsilon)_x\leq \varepsilon\eta(u^\varepsilon)_{xx}\qquad\quad(\varepsilon>0)\,.\]
Assuming further that the viscous solutions $u^\varepsilon$ do converge in a sufficiently strong sense
to $u(t,x)$, we obtain an intrinsic selection criterion for \eq{cons}: a weak solution $u(t,x)$ of \eq{cons} 
is said to be {\em admissible} provided it satisfies the entropy inequality
\beq\label{entr_ineq}
	\eta(u)_t+q(u)_x\leq 0 \qquad\text{(distributional sense)}
\eeq
whenever $(\eta,q)$ is a convex entropy pair \cite{lax1}.
While the issue of admissibility criteria is far from settled for multi-dimensional 
problems (see \cites{daf, delsze, ell}), entropies are of central importance in the 
theory of hyperbolic conservation laws.

\subsection{Extensions and entropies for systems of conservation laws}
We start by making the following (not entirely standard) definition:
\begin{definition}\label{ext_entr}
	Let \eq{cons} have a smooth flux function $f:\Omega\to\RR^n$.
	A smooth scalar field $\eta:\Omega\to\RR$ is an 
	{\em extension} for \eq{cons} provided the map $u\mapsto \nabla \eta (u)Df(u)$ is the 
	$u$-gradient of a scalar field $q(u)$. If so, $q(u)$ is 
	the {\em flux} of the extension $\eta$. 
	An extension $\eta(u)$ is a (strict) {\em entropy} for \eq{cons} provided it is a (strictly) 
	convex function of the conserved quantities, i.e.\ the Hessian $D^2\eta(u)$ is (strictly) 
	positive semi-definite on $\Omega$.   
\end{definition}
\noindent
The following facts about extensions and entropies are well-known, \cites{boi,daf,ser1,lax1,lax2,mo,god,ha,gr,bgs,cl,ts1}:
\begin{enumerate}
	\item Any system \eq{cons} is equipped with {\em trivial extensions}, i.e.\ affine maps 
	\beq\label{triv_ext}
		\eta(u)=a\cdot u+b\,,\qquad \qquad a\in\RR^n,\, b\in\RR
	\eeq
	with corresponding fluxes $q(u)=a^Tf(u)+c$, $c\in\RR$. 
	\item For a single equation ($n=1$) any scalar field is an extension.
	\item Strictly hyperbolic rich systems (Section \ref{rich}), and in particular strictly hyperbolic 
	systems of two conservation laws ($n=2$), possess large (i.e.\ ``rich") families of extensions.  
	More precisely, given a strictly hyperbolic, rich system \eq{cons},
	a base point $\bar u\in \Omega$, and a choice of $n$ functions of one variable each; then there 
	is an extension $\eta$ that reduces to each of the given functions along each Riemann coordinate 
	curves through $\bar u$. See \cites{cl,ser2,ts1} and Theorem \ref{rich_unconstr} below.	
	\item	An extension for \eq{cons} with $n\geq 3$ must satisfy an overdetermined system of differential 
	constraints. Consequently one would expect that a ``randomly" chosen system \eq{cons} would not 
	possess many, if any, non-trivial extensions.  
	On the other hand, many systems appearing in physical applications {\em are} equipped 
	with extensions/entropies. The prime example is given by the Euler system for a medium 
	with a convex equation of state (see Example \ref{gas_dyn} below).
	\item (Friedrichs-Lax \cite{fl}) If \eq{cons} is equipped with a strict entropy $\eta$, then it is Friedrichs
	symmetrizable: pre-multiplication of \eq{cons} by $D^2\eta$ yields a quasi-linear system 
	with symmetric coefficient matrices. This implies energy estimates and the Cauchy problem for such 
	systems is well-posed in appropriate Sobolev spaces \cite{bgs}.
	\item (Mock \cite{mo}, Godunov \cite{god}) If \eq{cons} is equipped with a strict entropy 
	$\eta$, then the change of variables $u\mapsto v:=\nabla_u\eta$ transforms \eq{cons} 
	into a symmetric, conservative system in gradient form: 
	$\left[\nabla \phi(v)\right]_t+\left[\nabla  \psi(v)\right]_x=0$. 
	Conversely, if \eq{cons} is (``conservatively") symmetrizable in this sense, then it possesses
	a strict entropy. 
	\item The classes of extensions and entropies for the Euler system of compressible gas dynamics, 
	have been determined (for general equations of state). See Example \ref{gas_dyn} and \cites{hllm,rj,ha,ser1}.
	\item The entropy inequality \eq{entr_ineq} is related to other types of selection criteria such as the Lax-, Liu-, and 
	entropy-rate criterion; see \cite{daf} for an overview. 
	\item For strictly hyperbolic systems there is a general relationship 
	between interaction coefficients, eigenvalues and eigenvectors of the system; see \cite{sev} 
	and Observation \ref{obs_sev} below.
\end{enumerate}
The main goal of the present work is to gain a better understanding of ``how many"
extensions a system of conservation laws has. Except for the case of strictly hyperbolic rich systems, 
we are not aware of general results that provide detailed information about the size of the set of extensions.
In this article we provide a geometric approach to this problem and consider the lengths of eigen-vectors as 
measured with respect to the inner product determined by the extensions, rather than the extensions 
themselves, as the primary unknowns. This leads to a natural formulation in terms of flat connections and 
opens up for the application of standard integrability theorems. 

To give a precise formulation of the problem we first recall how the requirement of 
possessing an extension places restrictions on \eq{cons}. By definition \eq{cons} is equipped with 
an extension $\eta$ if and only if 
\beq\label{cond1}
	\del_i\big[\big(\nabla \eta Df\big)_j\big]=\del_j\big[\big(\nabla \eta Df\big)_i\big]\qquad\forall i<j
	\qquad\qquad(\del_i= \frac{\del}{\del u^i})\,.
\eeq
Performing the differentiations we get the equivalent conditions that
\beq\label{cond2}
	\del_i\big(\nabla\eta \big)\cdot \del_j f=\del_j\big(\nabla\eta \big)\cdot \del_i f\qquad\forall i<j\,.
\eeq
Thus, an extension must satisfy $\frac{n(n-1)}{2}$ second order (linear, variable-coefficients) differential equations. 
In general these cannot be satisfied in a nontrivial manner (i.e.\ by a nonlinear field $\eta$) unless $n=1$ or $n=2$.

We assume that the system \eq{cons}  is hyperbolic with a basis $\{R_i(u)\}_{i=1}^n$ 
of right eigenvectors of $Df$ with corresponding eigenvalues $\{\lam^i(u)\}_{i=1}^n$. 
The requirements \eq{cond2} are then equivalent to the requirement that the matrix $D^2\eta\,Df$ is symmetric and therefore,
\[\big(R_i^TD^2\eta\big)\big(Df R_j\big)=\big(R_j^TD^2\eta\big)\big(Df R_i\big)\qquad\forall i<j\,,\]
Since $Df R_i=\lambda^i\, R_i$ for $i=1,\dots, n,$ we thus require 
\beq\label{cond4}
	\text{for each pair $1\leq i\neq j\leq n$, either \quad 
		$\lambda^j= \lambda^i$ \quad or \quad $R_i^T\big(D^2\eta\big) R_j = 0$}\,.
\eeq
Also, by hyperbolicity, convexity of $\eta$ is equivalent to 
\beq\label{cond5}
	R_i^T\big(D^2\eta\big) R_i\geq 0\qquad\forall i=1,\dots,n\,.
\eeq
Summing up we have:
\begin{proposition}\label{prelim}
	Given a hyperbolic system \eq{cons} such that $Df$ has right eigenvectors $\{R_i(u)\}_{i=1}^n$ 
	and eigenvalues $\{\lam^i(u)\}_{i=1}^n$. Then $\eta$ is an extension for \eq{cons} if and 
	only if \eq{cond4} holds, and it is an entropy if and only if \eq{cond4} and \eq{cond5} hold.
	If \eq{cons} is {\em strictly} hyperbolic (i.e., $\lam^i(u)\neq\lam^j(u)$ 
	for all $u\in\Omega$), then $\eta$ is an extension for \eq{cons} if and only if $\{R_i(u)\}_{i=1}^n$ 
	is orthogonal with respect to the inner product $D^2\eta$. 
\end{proposition}
%


\subsection{Discussion and outline} 
It is clear from Proposition 
\ref{prelim} that  the eigen-frame $\mathfrak R:=\{R_1,\dots,R_n\}$ plays a central role in admitting 
or preventing non-trivial extensions for \eq{cons}. Our main goal is to analyze, {\em in terms 
of the frame $\mathfrak R$}, how large the class of extensions is. 
We will therefore prescribe the frame $\mathfrak R$, and then determine how many scalar fields $\eta$ 
have the property that $\mathfrak R$ is orthogonal with respect to $D^2\eta$. As detailed below,
this problem leads to an over-determined algebraic-differential system, which we call the ``$\beta$-system."

\begin{remark}
	The $\beta$-system seems to be derived for the first time by Conlon and Liu (\cite{cl}, Section 2). 
	These authors work in the setting of a given, strictly hyperbolic system, and record the fact
	that systems with a coordinate system of Riemann invariants have rich families of extensions
	and entropies. The latter fact was observed independently by Tsarev \cite{ts1}; see also \cite{ser2} 
	and Theorem \ref{rich_unconstr} below. We are not aware of further results on how many 
	extensions there are in the absence of a coordinate system of Riemann invariants.
\end{remark}

The unknowns in the $\beta$-system are the lengths of the frame-vectors $R_i$ as measured with 
respect to the inner-product $D^2\eta$. The Hessian matrix $D^2\eta$ is determined from these lengths 
and the given frame $\R$. In turn, the actual extensions $\eta$ are determined from $ D^2\eta$ by solving  
a system of $\frac{n(n+1)}{2}$ linear, second order PDEs, 
which amounts to successive integration of $n(n+1)$ first order linear ODEs (see Remark~\ref{eta_reconstr}).

Our primary concern in the present work is to analyze the size of the solution set of the $\beta$-system. 
(The issue of existence of {\em entropies}, i.e.\ whether the $\beta$-system admits solutions with 
all $\beta^i>0$, will be pursued elsewhere.)
After introducing an appropriate geometric setup we can employ standard integrability theorems 
(Frobenius, Darboux, Cartan-K\"ahler) to analyze the size of the solution set. The answer will specify
how many free constants and functions appear in a general solution, and thus provides an answer to
how ``rich" the class of extensions is. 
This geometric approach appears more convenient than a direct analysis of the second order system 
\eq{cond2} for the extensions $\eta$ themselves.

Of course, having prescribed the frame $\mathfrak R$, an immediate issue is whether there 
{\em are} systems \eq{cons} with $\mathfrak R$ as their eigen-frame. This question was studied 
by the authors in \cite{jk1} where it was formulated in terms of another over-determined
algebraic-differential system, the ``$\lambda$-system", whose unknowns are the eigenvalues-to-be in 
a corresponding system \eq{cons}. It is a non-obvious fact that there are frames $\mathfrak R$ with the 
property that the only systems \eq{cons} with eigen-frame $\mathfrak R$, are trivial systems
of the form $u_t+\bar\lam u_x=0$, $\bar\lam\in\RR$. See Examples \ref{rich-rank2} and \ref{ex_jk1_5.4}.

It is important to note that a given frame $\R$ can give rise to  a family of systems \eq{cons} 
which may include both strictly-hyperbolic and non strictly-hyperbolic systems.
From \eq{cond4} we see that for a strictly hyperbolic system of conservation laws with eigenframe $\R$, 
the solutions of  the $\bet$-system give rise to all possible extensions. On the other hand, for a non-strictly 
hyperbolic system \eq{cons} the solutions of the $\bet$-system may produce only a proper subset of all extensions,
since the $\bet$-system does not account for extensions that are due to coalescing eigenvalues. (See Example 
\ref{strict_hyp_not_all}  and Remark \ref{strict-hyp} for more details.)

The paper is organized as follows. 
In Section~\ref{problem}, we formulate the problem of finding extensions $\eta$ such that a given frame 
is orthogonal with repsect to the inner product $D^2\eta$ (see the second part of \eq{cond4}). 
This leads to the $\beta$-system in Section~\ref{beta_prob}.
In Section~\ref{lambda_prob}, we review briefly the problem of constructing conservative systems \eq{cons} 
with a prescribed eigen-frame (analyzed in \cite{jk1}), and record the corresponding $\lambda$-system. 
Section \ref{relatns} details the relationships between the two systems. Not surprisingly there are 
close connections, but also important differences, between the $\lam$- and the $\bet$-systems.  

In Section~\ref{geom-prelim} we  review relevant facts about connections on frame bundles and 
provide coordinate-free formulations of the $\lam$- and the $\bet$-systems.
This general framework reveals the geometric structure that underlies the problem.
This geometric structure immediately provides us with important identities on the coefficients of  
the $\lam$- and the $\bet$-systems that play important part in deriving compatibility conditions 
for these systems. It also readily shows how the systems  behave under changes of coordinates. 
This is useful for example when formulating the $\beta$-system for rich systems in the associated 
Riemann coordinates (see Section \ref{rich}).

It turns out that the analysis of the $\bet$-system is 
more involved than for the $\lambda$-system. 
We first consider the case when $\beta(\R)$ contains no algebraic part (Section \ref{no-alg}). 
The corresponding systems \eq{cons} are then necessarily rich. (This covers all $2\times 2$-systems
as well as rich systems with orthogonal eigen-frames.)
We then treat systems of 
three equations (Sections \ref{n=3} and \ref{proof}), and give a complete analysis for both rich 
and non-rich systems through a breakdown 
similar to what was done for the $\lambda$-system in \cite{jk1}.
Finally, in Section \ref{exs},
we consider several examples that illustrate our approach and results. 
In particular, we treat the case when the frame $\mathfrak R$ is that of the Euler system. 

We have developed 
{\sc Maple} code\footnote{Posted on \url{http://www.math.ncsu.edu/~iakogan/symbolic/geometry_of_conservation_laws.html}}
 to calculate solutions of the $\lambda$- and $\beta$-systems.

\section{Problem formulation; $\beta$-system and $\lambda$-system}\label{formulation1}
\subsection{Problem formulation}\label{problem}
By definition, a {\em frame} $\R$ on $\Omega$ is a collection of smooth vector fields $\{r_1,\dots,r_n\}$ which are linearly 
independent at each point of $\Omega$. Its dual \emph{co-frame} is $\mathfrak L:=\{\ell^1,\dots, \ell^n\}$ where the differential 
1-forms $\ell^i$ satisfy $\ell^i(r_j)=\delta^i_j$.

Let  $u^1,\dots, u^n$ be a fixed coordinate system on $\Omega$. By the {\em $u$-representation} of $\mathfrak R$ we mean the collection $\{R_1(u),\dots,R_n(u)\}$
of column vectors $R_i=[R_i^1,\dots,R_i^n]^T$ given by
\[ r_i\big|_u=\sum_{k=1}^nR_i^k(u)\pd{}{u^k}\Big|_u\qquad\qquad i=1,\dots,n\,.\]
We let $R(u),\, L(u)\in\RR^{n\times n}$ be given by
\beq\label{RL}
	R(u):=[R_1(u)\,|\,\cdots\,|\,R_n(u)]\qquad\mbox{and}\qquad L(u):=R^{-1}(u)
	=\left[\!\begin{array}{c} L^1(u)\\ \hline \vdots\\\hline L^n(u)\end{array}\!\right]\,.
\eeq
Then
\beq\label{1-frms}
	\ell^i\big|_u=\sum_{m=1}^n L_m^i(u)du^m\big|_u\,,
\eeq
and we define the $(n\times 1)$-vector of 1-forms $\ell$ by
\beq\label{ell}
	\ell:=\left[\begin{array}{c}
	\ell^1\\
	\vdots\\
	\ell^n
	\end{array}\right]=Ldu\,.
\eeq
We refer to $L(u)$, or $\{L^1(u),\dots,L^n(u)\}$, as the {\em $u$-representation} of the 
dual co-frame $\mathfrak L$.

Given $\R$ we  would like to find all scalar fields $\eta(u)$ that satisfy the second condition in \eq{cond4}:
\beq\label{orth-cond}
	R_i(u)^T\big(D^2_u\eta(u)\big) R_j(u) = 0 \quad\text{for each pair $1\leq i\neq j\leq n$.}
\eeq
If $\R$ is an eigen-frame for the flux $f$ of a conservative system \eq{cons},
then such an $\eta$ is an extension of \eq{cons}.  

 \begin{problem}\label{prob1-geom}
	Let $\mathfrak R$ be a given frame on $\Omega$. Find all scalar fields $\eta(u)$ 
	defined on a  neighborhood of a point $\bar u\in\Omega$, such that the frame $\R$ is orthogonal 
	with respect to the inner product defined by the Hessian matrix $D^2_u\eta$.
\end{problem}	
\begin{remark}
Note that if $D^2_u\eta$ is strictly positive definite, then it defines  a  positive definite (Riemannian) metric on an 
open subset of $\RR^n$, that can be locally expressed 
as the Hessian of a smooth function. Such metrics are called {\em Hessian metrics} (see Remark~\ref{bet-free} 
for a coordinate free  definition of a Hessian metric).  
They were introduced as real analogs of K\"ahlerian metrics on complex manifolds and 
have been extensively studied, see \cite{hirohiko} and references therein. 
The problem stated above, however, does not seem to appear in the literature.
\end{remark}
We proceed to show that  an extension $\eta$ is completely determined by the values 
\beq\label{bet_def}
	\bet^i(u):=R_i(u)^T\big(D^2_u\eta(u)\big) R_i(u) \qquad i=1,\dots, n,
\eeq
and reformulate our problem accordingly.
From \eq{orth-cond} and \eq{bet_def} we have that
$$R^T\big(D^2_u\eta\big) R=\diag[\beta^1,\dots,\beta^n]\,.$$
Recalling \eq{RL}, we thus obtain an equivalent condition that
\beq
\label{cond-eta}
	D^2_u\eta=L^T\,\diag[\beta^1,\dots,\beta^n] \,L\,.
\eeq
We now recall two well-known consequences of Poincar\'e's Lemma \cite{sp1}. 
The first result will be used immediately, while the second is used in the following sections.
\begin{proposition}\label{symm}
	A $u$-Jacobian on $\Omega$ is symmetric if and only if it is a $u$-Hessian.
\end{proposition} 
\begin{proposition}\label{jac_cond} 
	An $n\times n$-matrix $A(u)$ defined on  $\Omega$ is a $u$-Jacobian if and only if 
	\beq\label{jac}
		dA(u)\wedge du=0\,.
	\eeq
\end{proposition} 
Combining \eq{cond-eta} with Proposition~\ref{symm} we conclude that $\eta$ is an extension for 
a conservative system with eigen-frame $\R$ provided
\beq\label{cond6}
	L^T\diag[\beta^1,\dots,\beta^n]L \qquad \mbox{is a $u$-Jacobian.}
\eeq
Considering the frame $\R$ as given, we view \eq{cond6} as a condition on $\beta^1,\dots,\bet^n$, 
and obtain the following reformulation of Problem~\ref{prob1-geom}. 

\begin{problem}\label{prob1-entropy}
	Let $\mathfrak R$ be a given frame defined near $\bar u\in \Omega$, and let 
	the matrix $L(u)$ be the $u$-representation of the dual co-frame $\mathfrak L$, as in \eq{RL}. 
	Find $n$ scalar fields $\bet^1(u),\dots,\bet^n(u)$ defined on a neighborhood 
	$\mathcal U\subset \Omega$ of $\bar u$ such that with 
	\beq\label{B}
		\Beta:=\diag[\beta^1,\dots,\beta^n]\,,
	\eeq
	the symmetric matrix
	\beq\label{A}
		L^T(u)\Beta(u)L(u)\qquad
		\text{is the $u$-Jacobian of some map $\Psi:\mathcal U \to\mathbb R^n$.} 
		\eeq
	We are further interested in how large the set of solutions is, i.e.\ how many arbitrary 
	constants and functions appear in a general solution $\bet^1(u),\dots,\bet^n(u)$.
\end{problem}
In Section~\ref{beta_prob} we derive a system of differential and algebraic equations that is 
equivalent to condition \eq{A}. The following remark outlines how one can recover the solution 
of Problem~\ref{prob1-geom} from the solution of Problem~\ref{prob1-entropy}.
\begin{remark} \label{eta_reconstr} Condition \eq{cond-eta} provides the link between 
Problem~\ref{prob1-geom} and Problem~\ref{prob1-entropy}.   
The right-hand side of  \eq{cond-eta} can be computed from a given frame 
$\R$ and $\beta^1,\dots,\bet^n$. By symmetry of the matrices involved, 
\eq{cond-eta} provides $\frac {n(n-1)} 2$ 
linear second order PDEs for the scalar field $\eta$. 

Vice versa, given the solution of Problem~\ref{prob1-entropy}, we can recover $\eta$ by 
successively solving $n^2+n$ ODEs.
Indeed, $L^T(u)\Beta(u)L(u)$ is the Jacobian of a map 
$\Psi=(\Psi^1,\dots, \Psi^n):\mathcal U \to\mathbb R^n$.
This means that the $i$-th row of $L^T(u)\Beta(u)L(u)$ is the gradient of the function 
$\Psi^i$, $i=1,\dots, n$. Recovering a function of $n$ variables from its gradient 
requires successive integration of $n$ ODEs. Therefore, we can recover
$\Psi=(\Psi^1,\dots, \Psi^n)$ in \eq{A} by solving $n^2$ ODEs. Also, by Proposition \ref{symm} we have that 
$$ (\Psi^1,\dots, \Psi^n)=\nabla\eta\qquad\text{for some scalar field $\eta$.}$$
From this we recover $\eta$ by successively solving $n$ ODEs.
Thus, all in all we can recover $\eta$ from  $L^T(u)\Beta(u)L(u)$ by integrating
of $n^2+n$ ODEs. Finally, a function is  recovered from its gradient uniquely up to an 
additive constant. Therefore, the above steps allow us to uniquely recover $\eta$ from 
$\beta^1,\dots,\beta^n$, up to addition of a trivial extension \eq{triv_ext}. Examples are given  
in Section \ref{exs}.

\end{remark}

\subsection{The $\beta$-system}\label{beta_prob}
 Applying Proposition \ref{jac_cond}, 
we conclude that $L^T\Beta L$ is a Jacobian (condition \eq{A}) if and only if  
 $\beta=(\beta^1,\dots,\beta^n)$ satisfies  the following {\em $\beta$-system}:
\beq\label{m_beta}
	\beta(\mathfrak R)\,:\qquad\qquad  d[L^T\Beta L]\wedge du = 0
	\qquad \qquad\text{($\Beta$ given by \eq{B}).}
\eeq
We proceed to rewrite \eq{m_beta} as an algebraic-differential system.
Define 
\beq\label{christoffel}
	\Gamma_{ij}^k:=L^k(D R_j)R_i \qquad\mbox{and}\qquad c_{ij}^k:=\Gamma_{ij}^k-\Gamma_{ji}^k\,
\eeq 
(see Section~\ref{geom-prelim} for  the geometric meaning of these functions).  
By applying the product rule to \eq{m_beta}, multiplying by $R^T$ from the 
left, and using that $(dL)R=-L(dR)$ we obtain that \eq{m_beta} is equivalent to
\beq \label{m-beta1}
	d\Beta \wedge Ldu=\left[\Beta L(dR)+\big(\Beta L(dR)\big)^T\right]\wedge L\,du\,,
\eeq 
which may be re-written as 
\beq \label{m-beta2}
	d\Beta\wedge \ell=\left[\Beta \mu+\big(\Beta \mu\big)^T\right]\wedge \ell,  
\eeq
where $\ell$ is given by \eq{ell} and the matrix $\mu$ is given by
\[\mu^k_j:=\big(LdR\big)_j^k=\sum_{i=1}^n \Gamma_{ij}^k \ell^i\,.\]
Applying \eq{m-beta2} to pairs of frame vectors $(r_i,r_j)$ we obtain an
explicit formulation of $\beta(\R)$:
\begin{eqnarray}
	r_i(\bet^j) = \bet^j\,(\Gamma_{ij}^j +c_{ij}^j)-\beta^i\, \Gamma_{jj}^i
	&\qquad& \mbox{for $i\neq j$,}\label{beta1}\\
	\bet^k\,c_{ij}^k +\beta^j \Gamma_{ik}^j-\beta^i\,\Gamma^i_{jk}=0
	&\qquad& \mbox{for $i<j$, $\eps(i,j,k)=1$,}\label{beta2}
\end{eqnarray}
where there are no summations. Note that \eq{beta1} gives $n(n-1)$ linear, homogeneous
PDEs, while \eq{beta2} gives $\frac{n(n-1)(n-2)}{2}$ algebraic relations. We observe that 
the left-hand side of \eq{beta2}  is skew-symmetric in $i$ and $j$, and that all coefficients 
$\Gamma_{ij}^k$ with $\eps(i,j,k)=1$ appear in \eq{beta2}.
We proceed with some simple but important properties of the $\beta$-system.
\begin{observation}(Trivial solutions)\label{trivial}
	We observe that $\beta(\mathfrak R)$ always has the trivial solution 
	$\beta^1=\cdots=\beta^n=0$, which corresponds to the trivial, affine 
	extensions in \eq{triv_ext}.
\end{observation}
\begin{definition}\label{non-deg}
	A solution $\beta=(\beta^1,\dots,\beta^n)$ of the $\beta$-system \eq{m_beta} 
	is called {\em non-degenerate} in $\Omega$ provided $\beta^i(u)\neq 0$ for all $i=1,\dots,n$, 
	$u\in\Omega$. 
\end{definition}
\begin{observation}(The effect of scaling)\label{scaling}
	Given smooth functions
	$\alpha^j:\Omega\to \mathbb R\setminus \{0\}$, $1\leq j\leq n$, we scale 
	the given frame and its dual frame according to:
	 \[\ti R_j(u)= \alpha^j(u)R_j(u)\qquad\text{and}\qquad \ti L^j(u):= {\alpha^j(u)}^{-1}L^j(u)\,.\] 
	Letting $\ti R:=R\alpha$, $\ti L:=\alpha^{-1}L$, 
	where $\alpha(u)=\diag[\alpha^1(u)\dots\alpha^n(u)]$, we may consider the $\beta$-system
	$\beta(\ti{\mathfrak R})$. Clearly, 
	\[L^T(u)\Beta(u) L(u) \quad\text{is a Jacobian if and 
	only if}\quad [\alpha(u)\ti L(u)]^T\Beta(u)\alpha(u) \ti L(u)\quad\text{is a Jacobian.}\] 
	Therefore, $\beta=(\beta^1,\dots,\beta^n)$ solves $\beta(\mathfrak R)$
	if and only if $\ti \beta=((\alpha^1)^2\beta^1,\dots,(\alpha^n)^2\beta^n)$ solves 
	$\beta(\ti{\mathfrak R})$. 
	However, while the solution set of the $\beta$-system is affected by scalings in this way, 
	the two systems $\beta(\mathfrak R)$ and $\beta(\ti{\mathfrak R})$ generate the 
	same class of extensions. This is clear from \eq{cond-eta} since $\ti L:=\alpha^{-1}L$.
\end{observation}

\subsection{The $\lambda$-system}\label{lambda_prob}
A solution of the $\bet(\R)$-system provides an extension for any conservative systems \eq{cons}
whose Jacobian $D_uf$ has eigen-frame $\R$.   
This raises a natural question: {\em What is the class of flux functions $f(u)$ with the property that 
the set of eigenvectors of $Df$ coincide with $\R$?} This  problem was treated by the authors in \cite{jk1}, and we 
briefly review some results from that paper. (A complete breakdown for the case $n=3$ is 
recorded in Proposition \ref{n=3lambda}.)

Given $n$ scalar fields $\lambda^1(u),\dots,\lambda^n(u):\Omega \to \RR$, we define 
the diagonal matrix  
\[\Lambda(u):=\diag[\lambda^1(u),\dots,\lambda^n(u)]\,.\]
The  {\em $\lambda$-system} $\lambda(\mathfrak R)$ associated with a frame $\R$  is  
the system of equations which encodes that the matrix $A(u):=R(u)\Lambda(u) L(u)$ is the 
Jacobian matrix of some map $f:\Omega\to\RR^n$. The set of such maps provide all 
possible flux functions $f$ in conservative systems \eq{cons} with $\mathfrak R$ as their 
eigen-frame. The unknowns of the $\lambda$-system are the eigenvalues 
$\lambda^1(u),\dots,\lambda^n(u)$ of the resulting Jacobian $D_uf$.
Proposition \ref{jac_cond} gives the following formulation of the $\lambda$-system:
\beq\label{m-lam}
		\nqquad\nqquad\lambda(\mathfrak R)\,:\qquad  d[R\Lambda L]\wedge du = 0\,.
\eeq
Evaluating \eq{m-lam} on pairs $(r_i,r_j)$ of frame vector fields we obtain the following
algebraic-differential system, whose coefficients are given by \eq{christoffel}:
\begin{eqnarray}
	r_i(\lam^j) &=& \Gamma_{ji}^j (\lam^i-\lam^j)
	\qquad \mbox{for $i\neq j$,}\label{sev1}\\
	(\lam^i-\lam^k)\Gamma_{ji}^k &=& 
	(\lam^j-\lam^k)\Gamma_{ij}^k
	\qquad \mbox{for $i<j$, $\eps(i,j,k)=1$,}\label{sev2}
\end{eqnarray}
where there are no summations. Note that \eq{sev1} gives $n(n-1)$ linear, homogeneous
PDEs, while \eq{sev2} gives $\frac{n(n-1)(n-2)}{2}$ linear algebraic relations, with coefficients
Concerning scaling of $\R$ (cf.\ Observation \ref{scaling}) we have that the solution set 
of the $\lambda$-system is unaffected by scaling. Indeed, since $\ti R\Lambda \ti L\equiv R\Lambda L$,
$\lam$ solves $\lambda(\ti{\mathfrak R})$ if and only if $\lam$ solves $\lambda(\mathfrak R)$.

As is clear from \eq{m-lam}, $\lam(\mathfrak R)$ always 
has a one-parameter family of {\em trivial solutions} given by 
$\Lambda(u)\equiv\bar\lambda I_{n\times n}$, $\bar\lambda\in\RR$, corresponding to diagonal 
affine fluxes 
\beq\label{trivial_flux} 
	f(u)=\bar \lambda u+\bar u.
\eeq 
There are frames for which the $\lam$-system has only trivial solutions, and there are frames  
for which the $\lam$-system has a large family of non-trivial solutions; see \cite{jk1}.

\begin{remark}\label{strict-hyp}
Note that even when a frame $\R$ admits strictly hyperbolic fluxes (i.e.~there is a solution of $\lam(\R)$ 
such that $\lam^1(u),\dots,\lam^n(u)$  are distinct $\forall u\in \Omega$), it also gives rise to  
non-strictly hyperbolic systems - in particular the trivial fluxes \eq{trivial_flux}.
The $\bet$-system \eq{m_beta} produces all extensions for strictly hyperbolic fluxes that correspond to $\R$, but, 
in general, it does not produce  all extensions for corresponding non-strictly hyperbolic fluxes. 
The extensions for fluxes with coalescing eigenvalues, which satisfy the first part of \eq{cond4}, 
are not necessarily covered by the $\bet$-system.  
\end{remark}

\subsection{Relation between the $\lam$- and the $\bet$-systems.}\label{relatns}
By restricting ourselves to searching for those extensions $\eta$ 
that satisfy \eq{cond4} due to orthogonality, we have separated the problem of finding fluxes that correspond to
the frame $\R$, from the problem of finding extensions that correspond to the same frame.  
Indeed,  no unkonwn of
$\beta(\mathfrak R)$ appears in $\lambda(\mathfrak R)$ or vice versa. The two systems can
therefore be solved independently of each other.

We note that both systems $\beta(\mathfrak R)$ and $\lam(\mathfrak R)$ encode the property 
that a matrix is a $u$-Jacobian. Also, the coefficients of both systems  are computed from the $u$-representation 
of the given frame $\R$. Nonetheless, we will show in Section~\ref{n=3} that there is in general no relationship
between the sizes of the solution sets of $\beta(\mathfrak R)$ and $\lam(\mathfrak R)$.
We observe that {\em orthogonal} frames provide a notable exception:
\begin{observation}(Orthonormal frames)\label{orthonormality}
	Comparing the $\beta$-system \eq{m_beta} and the $\lambda$-system \eq{m-lam} 
	we see that if $\mathfrak R$ is orthonormal relative to the standard
	inner product, then $R=L^T$ and the two systems coincide. By Proposition \ref{symm} 
	we deduce that the corresponding conservative systems \eq{cons} 
	are {\em gradient systems}: their flux function is the gradient of an extension,
	$f=(\nabla \eta)^T$.
\end{observation}
When the frame $\R$ is orthogonal the solutions to the $\lam$- and $\bet$-systems are 
related by scaling, and  the $\beta$-system necessarily have non-trivial solutions: 
\begin{observation}(Orthogonal frames)\label{orthogonal}
	Assume $\R$ is an orthogonal frame relative to the standard inner product, and let 
	$\alpha^i=|R_i|^{-1}$, such that $\ti\R=\{\alpha^1R_1,\dots,\alpha^nR_n\}$ is the 
	orthonormal scaling of $\R$. From Observations~\ref{scaling}  and~\ref{orthonormality}, 
	and the fact that the solution set of $\lam(\ti\R)$ coincides with that of $\lam(\R)$, we obtain:
	\beq\label{scale_rel}
		\text{$\beta^1,\dots,\beta^n$ solves $\beta(\R)$ if and only if $(\alpha^1)^2\,\beta^1,\dots,(\alpha^n)^2\,\beta^n$ solves 
		$\lam(\R)$.}
	\eeq 
	It follows from this that $\beta(\R)$ has non-trivial solutions. Indeed, $\lam(\R)$ has a 
	one-parameter family of trivial solutions $\lam^1=\dots=\lam^n\equiv\bar\lam\in\RR$. 
	However, as solutions to $\beta(\ti\R)\equiv \lam(\ti \R)$, these are non-trivial solutions. 
	They provide the extensions $\eta(u)=\textstyle \frac{1}{2}\bar\lam\,|u|^2$, which
	(according to the last part of Observation \ref{scaling}) are also extensions corresponding
	to the original, un-scaled frame $\R$. 
	In particular, we obtain the well-known fact that if \eq{cons} is a gradient system with flux
	$f=(\nabla \eta)^T$, then it has $\eta$ as an extension and $|u|^2$ as a strict entropy
	(\cites{fl, god, daf}). 
	
	We finally note that {\em any} other solution of $\lam(\R)$ also provides an extension of the 
	original gradient system $u_t+f(x)_x=0$ (see Example~\ref{rich_rank0}).
\end{observation}

If $n=2$ neither the $\bet$-system nor the $\lam$-system has an algebraic part.
For $n=3$ we shall see that the algebraic parts \eq{sev2} and \eq{beta2} have
	the same rank (Proposition \ref{n=3_ranks}). 
	However, Example~\ref{unequal_rank} illustrates that equality of ranks of the algebraic parts 
	does {\em not} generalize to $n\geq 4$. The rank zero case provides an exception in this respect:
\begin{observation}\label{zero-rank}
	The following three statements  are equivalent
	\begin{enumerate}
	\item $\Gamma^i_{jk}=0$ whenever $\eps(i,j,k)=1$.
	\item There are no algebraic conditions in the $\lam$-system \eq{sev1}-\eq{sev2}.
	\item There are no algebraic conditions in the  $\bet$-system \eq{beta1}-\eq{beta2}. 
	\end{enumerate} 
	
\end{observation}

When the algebraic parts \eq{beta2} and \eq{sev2} of the $\bet$-system and the $\lam$-system are 
non-trivial, we can easily derive, {\em under the assumption of strict hyperbolicity}, the following relationship 
between  solutions of the $\bet$-system and the $\lam$-system (see Proposition 8 in \cite{sev}): 
\begin{observation}\label{obs_sev}
Under the assumption of strict hyperbolicity:
\beq\label{lam_bet_reln}
	\frac{c_{jk}^i\beta^i}{\lam^j-\lam^k}+\frac{c_{ki}^j\beta^j}{\lam^k-\lam^i}+
	\frac{c_{ij}^k\beta^k}{\lam^i-\lam^j}=0\qquad \text{whenever $i<j<k$.}
\eeq
In the present setup this may be deduced directly from the algebraic 
parts of the $\lam$- and $\beta$-systems:
replace $\Gamma_{ji}^k$ with $\Gamma_{ij}^k-c_{ij}^k$
in the algebraic part of $\lam$ system \eq{sev2}, and 
solve for $\Gamma_{ij}^k$
\[\Gamma_{ij}^k=\frac{\lam^i-\lam^k}{\lam^i-\lam^j} c_{ij}^k\,,\qquad \epsilon(i,j,k)=1.\]
Substituting 
\[\Gamma_{ik}^j=\frac{\lam^i-\lam^j}{\lam^i-\lam^k} c_{ik}^j
\qquad\text{and}\qquad \Gamma_{jk}^i=\frac{\lam^j-\lam^i}{\lam^j-\lam^k} c_{jk}^i\]
in the algebraic part \eq{beta2} of the $\beta$-system yields \eq{lam_bet_reln}.

\end{observation}

From the point of view of constructing relevant systems \eq{cons} from a given frame 
$\mathfrak R$, one would first solve $\lambda(\mathfrak R)$ and 
then proceed to solve $\beta(\mathfrak R)$, provided the former possesses non-trivial 
solutions. However, to highlight the fact that the solution sets of $\lam(\R)$ and 
$\bet(\R)$ may be of different size, we have also included Example~\ref{ex_jk1_5.4}. 
The two frames in that example show that $\lambda(\mathfrak R)$ may have only trivial 
solutions while $\beta(\mathfrak R)$ possesses non-trivial solutions.

\bigskip
\subsection{Connections on frame bundles and a coordinate free formulation}
\label{geom-prelim}
The coefficients $\Gamma_{ij}^k$ that appear in both the $\bet$- and the $\lam$-systems have 
a natural geometric interpretation as connection components (Christoffel symbols) of a flat and 
symmetric connection relative to the frame $\R$.
In this section we review the basic facts about connections on a frame bundle, which we will use 
to analyze the $\bet$-system. In Remark~\ref{bet-free} we give a coordinate-free definition of a 
Hessian metric and indicate how the $\bet(\R)$-system can be obtained in a coordinate-free 
manner. Similarly, Remark~\ref{lam-free} provides a coordinate-free definition of a Jacobian 
tensor and describes a coordinate-free approach to the $\lam(\R)$-system. 
A comprehensive account of the differential geometry material used in this section 
can be found in \cite{amf}, \cite{lee}, and \cite{hirohiko}.

Given an $n$-dimensional smooth manifold $M$
we let $\mathcal{X}(M)$ and $\mathcal{X}^*(M)$ denote the set of smooth  vector fields and 
differential 1-forms on $M$, respectively.
A \emph{frame} $\mathfrak R=\{r_1,\dots, r_n\}$  is 
a set of vector fields which span the tangent space $T_p M$ at each point $p\in M$. 
A \emph{coframe} $\{\ell^1,\dots, \ell^n\}$ is a set of $n$ differential 1-forms 
which span the cotangent space $T_p^* M$ at each point $p\in M$.  
The coframe and frame are \emph{dual} if $\ell^i(r_j)=\delta^i_j$ (Kronecker delta). If $u^1,\dots, u^n$ 
are local coordinate functions on $M$, then $\{\pd{}{u^1},\dots,\pd{}{u^n}\}$ is the corresponding local
\emph{coordinate frame}, while $\{du^1,\dots,du^n\}$ is the  dual local \emph{coordinate coframe}.
The \emph{structure coefficients} $c^k_{ij}$ of $\mathfrak R$ 
are defined by
\beq\label{lie-bracket}
	[r_i,r_j]=\sum_{k=1}^nc^k_{ij}\, r_k\,,
\eeq
and the dual coframe has related structure equations given by 
\beq\label{dl}
	d\ell^k=-\sum_{i<j} c^k_{ij}\, \ell^i\wedge \ell^j\,.
\eeq 
It can be shown (see Proposition 5.14 in \cite{sp1}) that there exist coordinate functions 
$w^1,\dots,w^n$ on an open subset 
of $\Omega$ such that $r_i=\pd{}{w^i},\,i=1\dots,n$, if and only if $r_1,\dots,r_n$ 
commute, i.e.\ all structure coefficients are zero. We recall the slightly weaker requirement of richness:
\begin{definition}\label{richness}
	The frame $\{r_i\}_{i=1}^n$ is {\em rich} provided its structure coefficients satisfy 
	\beq\label{rich_def}
		c^k_{ij}=0 \qquad\text{whenever $\eps(i,j,k)=1$.}
	\eeq
\end{definition}
\noindent
We note that a rich frame may be scaled so as to yield a commutative frame \cite{daf}.

An \emph{affine connection} $\nabla$ on $M$ is an 
$\RR$-bilinear map 
\[\mathcal{X}(M)\times \mathcal{X}(M)\to\mathcal{X}(M)\qquad\qquad
	(X,Y)\mapsto \nabla_X Y\]
such that for any smooth function $f$ on  $M$
\beq\label{connection}
	\nabla_{fX}Y=f\nabla_XY, \qquad \nabla_X (fY)=(Xf)Y+f\nabla_XY\,.
\eeq
By $\RR$-bilinearity and \eq{connection} the connection is uniquely defined by prescribing it 
on a frame:
$$\nabla_{r_i}r_j=\sum_{k=1}^n\Gamma^k_{ij}r_k,$$
where the smooth coefficients $\Gamma^k_{ij}$ are called \emph{connection  components, or Christoffel symbols,} 
relative to the frame $\{r_1,\dots,r_n\}$. Any choice of a frame and $n^3$ functions 
$\Gamma^k_{ij},\, i,j,k=1,\dots,n,$ defines an affine connection on $M$.
A change of frame induces a change of the connection components, and this 
change is not  tensorial. E.g., a connection with zero components 
relative to a coordinate frame, may have non-zero components relative to 
a non-coordinate frame.

A connection uniquely defines {\em a covariant derivative} $\nabla_X T$ of any tensor field $T$ on $M$ in the direction of a vector field $X$ (see for instance Section 1.2 of \cite{hirohiko}).

Given a frame $\{r_1,\dots, r_n\}$ with associated Christoffel symbols $\Gamma^k_{ij}$
and  the dual frame $\{\ell^1,\dots, \ell^n\}$, we define the \emph{connection 1-forms} $\mu_i^j$ by
\beq\label{conn-forms}\mu_i^j:=\sum_{k=1}^n\Gamma^j_{ki}\ell^k\,.\eeq 
In turn, these are used to define the \emph{torsion} 2-forms
\beq\label{tor}
	{\tor}^i:=d\ell^i+\sum_{k=1}^n\mu^i_k\wedge  \ell^k
	=\sum_{k<m}T^i_{km}\ell^k\wedge \ell^m,\qquad i=1,\dots,n\,,
\eeq
and the \emph{curvature} 2-forms
\beq\label{cur}
	\cur^j_i:=d\mu^j_i+\sum_{k=1}^n\mu^j_k\wedge\mu_i^k
	=\sum_{k<m}R^j_{i\,km} \ell^k\wedge \ell^m\,.
\eeq
The second equalities of \eq{tor} and \eq{cur} define the components of 
torsion and curvature tensors, respectively:
\begin{eqnarray}
	\label{T1}T^i_{km}&=&\Gamma^i_{km}-\Gamma^i_{mk}-c^i_{km}\\
	\label{R} R^j_{i\,km}&=&r_k\big(\Gamma^j_{mi}\big)
	-r_m\big(\Gamma^j_{ki}\big)+\sum_{s=1}^n\big(\Gamma^j_{ks}
	\Gamma^s_{mi}-\Gamma^j_{ms}\Gamma^s_{ki}-c^s_{km}
	\Gamma^j_{si}\big).
\end{eqnarray} 
We can write equations \eq{tor} and \eq{cur} in the compact matrix 
form
\beq\label{tor-cur} 
	\tor=d\ell+\mu\wedge \ell,\quad \cur=d\mu+\mu\wedge\mu
\eeq
where $\tor=(\tor^1,\dots,\tor^n)^T$, and 
$\cur$ and $\mu$ are the matrices with components $\cur^j_i$  and $\mu^j_i$, respectively.
The connection is called  \emph{symmetric} if the torsion form is identically zero and  
it is called  \emph{flat}  if the curvature form is  identically zero. Equivalently:
\beq\label{cur0-tor0}	
	d\ell=-\mu\wedge \ell\qquad\mbox{(Symmetry)},
	\qquad
	d\mu=-\mu\wedge\mu \qquad\mbox{(Flatness).}
\eeq
In terms of Christoffel symbols and structure coefficients this 
is equivalent to 
\beq\label{T0}
	c^i_{km} = \Gamma^i_{km}-\Gamma^i_{mk}\qquad\mbox{(Symmetry)}
\eeq
and
\beq\label{R0}
	r_m\big(\Gamma^j_{ki}\big)-r_k\big(\Gamma^j_{mi}\big)
	=\sum_{s=1}^n \big(\Gamma^j_{ks}\Gamma^s_{mi}-\Gamma^j_{ms}
	\Gamma^s_{ki}-c^s_{km}\Gamma^j_{si}\big)\qquad\mbox{(Flatness)}.
\eeq 
One can also show that a connection $\nabla$ is symmetric and flat if and only if in 
a neighborhood  of each point there exist coordinate functions
 $u^1,\dots,u^n$ with the property that the 
Christoffel symbols relative to the coordinate frame are zero: 
\beq\label{affine}
	\nabla_{\pd{}{u^i}}\pd{}{u^j}=0\qquad \mbox{for all $i,j=1,\dots,n$.}
\eeq
Such a coordinate system is called an {\em affine coordinate system} with respect to $\nabla$. 
Manifolds with a symmetric and flat connection are called \emph{flat}.

\begin{remark}\label{bet-free}{(\sc  Coordinate-free definition of a Hessian and the $\bet$-system)} 
In \cite{hirohiko},   $g$ is called a Hessian metric on a  manifold $M$ with a {\em flat, symmetric} 
connection $\nabla$, if there exists a function 
$\eta\colon M\to\RR$ such that
\beq\label{hessian_free}
	g=\nabla d\eta,
\eeq 
Explicitly \eq{hessian_free} says 
that for  any two vector fields $X,Y\in\mathcal X(M)$:
\beq\label{hessian_free_detailed}
	g(X,Y)=(\nabla_Xd \eta)(Y):=X(d\eta(Y))-d\,\eta(\nabla_X Y).
\eeq
The advantage of the condition \eq{hessian_free} is that it provides us with a coordinate free definition 
of Hessian metrics, whereas a definition based on Hessian matrices requires a choice of coordinates. 
The pair $(\nabla,g)$ is called a {\em Hessian structure} on $M$.  Applying \eq{hessian_free_detailed} 
to an affine coordinate frame $\pd{}{u^i}$, $i=1,\dots,n,$ we can verify that 
$g(\pd{}{u^i},\pd{}{u^j})=\frac{\del^2\eta}{\del u^i\,\del u^j}$. If $\R=\{r_1,\dots,r_n\}$ is a frame, whose 
Christoffel symbols relative to $\nabla$ are $\Gamma_{ij}^k$, then a simple computation shows that 
\beq\label{gr}
	g(r_i,r_j)=r_i(r_j (\eta))-\sum_{k=1}^n\Gamma^k_{i,j}\,r_k(\eta).
\eeq
We observe  from \eq{gr}  that the symmetry of the metric, $g(r_i,r_j)=g(r_j,r_i)$, is equivalent 
to the commutator identity $[r_i,r_j]\eta=\sum_{k=1}^nc^k_{ij}\,r_k(\eta).$

It is  shown in \cite{hirohiko} that a pair $(\nabla,g)$ of a flat, symmetric  connection $\nabla$ 
and a metric $g$, is a Hessian structure if and only if it satisfies the Codazzi equations
\beq\label{codazzi} (\nabla_X g)(Y, Z)= (\nabla_Y g)(X, Z), \quad \forall X,Y, Z \in \mathcal X(M).\eeq
Let us further assume that $\R$ is an orthogonal frame relative to $g$, i.e. $g(r_i,r_j)=\delta^i_j \beta^i$. 
Recalling that $(\nabla_X g)(Y, Z):=X(g(Y,Z))-g(\nabla_XY,Z)-g(Y,\nabla_X Z)$, and substituting
$X=r_i$, $Y=r_j$ with $i\neq j$, and $Z=r_k$ in \eq{codazzi}, we obtain the $\bet(\R)$-system 
\eq{beta1}-\eq{beta2}.

 \end{remark}

\begin{remark}({\sc  Coordinate-free definition of a Jacobian and the 
$\lam$-system})\label{lam-free}  
We can similarly provide a coordinate-free definition of a Jacobian. 
For a manifold $M$ with a {\em flat, symmetric} connection $\nabla$, a linear map 
$J\colon \mathcal X(M)\to \mathcal X(M)$ is called  a {\em Jacobian map}
if  there exists a vector field  $V\in\mathcal X(M)$ such that 
\beq\label{jacobian_free} 
	J=\nabla V \quad \Leftrightarrow \quad J(X)=\nabla_X V, \quad \forall X\in \mathcal X(M). 
\eeq
If \eq{jacobian_free} holds we 
say that $J$ is the {\em Jacobian of $V$} and use the notation $J_V$. 
The flatness and symmetry of $\nabla$ implies that for $\forall X,Y\in \mathcal X(M)$
\beq\label{j-cond}
	\nabla_XJ_V(Y)-\nabla_YJ_V(X)
	:=\nabla_X(\nabla_Y(V))-\nabla_Y(\nabla_X(V))
	=\nabla_{[X,Y]} V=:J_V([X,Y]).
\eeq
Let $(u^1,\dots, u^n)$ be an affine system of coordinates (see \eq{affine}) and 
$V=\sum_{i=1}^n f^i(u)\pd{}{u^i}$, then according to \eq{jacobian_free}
$$J_V(\pd{}{u^j})=\sum_i^n\pd{f^i}{u^j}\,\pd{}{u^i},$$ which is exactly $j$-th column 
vector of the usual Jacobian matrix of  a vector valued function $f(u)=(f^1,\dots,f^n)$.
If $\R=\{r_1,\dots,r_n\}$ is a frame whose Christoffel symbols relative to $\nabla$ are 
$\Gamma_{ij}^k$ and $V=\sum_{i=1}^n \tilde f^i\,r_i$ then \eq {jacobian_free} implies:
\beq\label{Jr}
	J_V(r_j)=\sum_{i=1}^n\left[r_j (\tilde f^i)
	+\sum_{k=1}^n\Gamma^i_{jk}\,\tilde f^k\right]\,r_i.
\eeq
Let us further assume that $\R$ is  a set of eigenvector-fields of $J$  with real-valued 
eigenvalues $\lam^i$:
$$J(r_i)=\nabla_{r_i} V=\lambda^i r_i\,.$$
Substitution of $X=r_i$ and $Y=r_j$, with $i\neq j$, into \eq{j-cond} produces the 
$\lam(\R)$-system \eq{sev1}-\eq{sev2}.


\end{remark}

Before proceeding with the analysis of the $\beta$-system we return to the setting of 
Sections \ref{problem}-\ref{lambda_prob} where the coordinate system $(u^1,\dots,u^n)$ and  
the frame $\R=\{r_1,\dots,r_n\}$ on $\Omega$ are fixed. 
From now on $\nabla$ will denote the unique flat and symmetric connection 
satisfying \eq{affine}. As indicated at the beginning of Section \ref{geom-prelim} the coefficients 
$\Gamma^k_{ij}$ defined in \eq{christoffel}, and appearing in the $\lam$- and $\bet$-systems, 
are then the Christoffel symbols of the connection $\nabla$ relative to the frame $\R$. Also, the 
notation for the coefficients $c_{ij}^k$ in \eq{christoffel} is consistent with symmetry \eq{T0} of $\nabla$.

\medskip

\section{Analysis of $\beta(\mathfrak R)$ with no algebraic part}\label{no-alg}
In this section we analyze  $\bet$-systems with no algebraic part. We show that the frames 
corresponding to such systems are necessarily rich. Examples in Section~\ref{rich-const} 
illustrate that the converse  is not true: there are rich systems whose
corresponding $\bet$-system imposes non-trivial algebraic constraints. It is true, however, that the 
$\bet$-system for a rich frame that admits a {\em strictly hyperbolic} flux has no algebraic constraints.

In \cite{jk1} we analyzed $\lambda(\R)$ for general rich frames $\R$. When the algebraic 
part is non-trivial this analysis is complicated due to possible additional algebraic constraints 
imposed by the differential part of $\lambda(\R)$. A similar breakdown does not seem feasible 
for $\beta(\R)$ unless $n\leq 3$. In this section we therefore treat rich systems of any size, but 
with trivial algebraic part. Rich systems of three equations with algebraic constraints are covered in 
Section~\ref{rich_rank=1_n=3}. 

\subsection{Absence of the algebraic part implies richness}\label{rich}
Given a frame $\R$ such that $\bet(\R)$ does not impose any algebraic constraints.
According to Observation~\ref{zero-rank} this is the case if and only if the same is 
true for the $\lam$-system, and occurs if and only if
\beq\label{no_alg-G}
	\Gamma_{ij}^k=0 \qquad\text{whenever $\eps(i,j,k)=1$.}
\eeq
The symmetry conditions \eq{T0} then imply that
\beq\label{no_alg-c}
	c_{ij}^k=0 \qquad\text{whenever $\eps(i,j,k)=1$},
\eeq 
whence, according to Definition \ref{richness}, the frame $\R$ is rich, and we have
\beq\label{rich-frame}
	[r_i,r_j]\in\text{span}\, \{r_i,r_j\}\qquad\text{for all $1\leq i,\, j\leq n$.}
\eeq
It follows from \eq{rich-frame}  that for every $i=1,\dots,n$, the vector fields 
$r_1, \dots, r_{i-1},r_{i+1},\dots,r_n$ are in involution. By Frobenius' theorem 
(Chapter 6 in \cite{sp1},  Section 7.3 in \cite{daf}) there exist scalar fields $w^i$,  $i=1,\dots,n$,
on an open neighborhood (again denoted by $\Omega$) of an arbitrary  point in $\Omega$, such that  
\[r_j(w^i)\left\{\begin{array}{ll}
	=0 & \mbox{if $i\neq j$}\,\\
	\neq 0 & \mbox{if $i=j$}\,
\end{array} \quad \forall u\in \Omega\,.\right.\]
By setting $\tilde r_i:= r_i/r_i(w^i)$ we achieve the normalization
\beq\label{norm}
	\tilde r_j(w^i)\equiv \delta^i_j\,\,\Longleftrightarrow \,\tilde r_i=\pd{}{w^i}.
\eeq
By Observation~\ref{scaling} we may assume, without loss of generality, that the given rich 
frame $\R$ satisfies \eq{norm} and  therefore is commutative. We denote the change of 
coordinates map by $\rho$: 
$$u\mapsto \rho(u)=(w^1(u),\dots,w^n(u)).$$
The $w$-coordinates are  referred to as {\em Riemann coordinates}.

\subsection{$\beta(\mathfrak R)$ in Riemann coordinates}\label{bet-rich}
In Riemann coordinates the $\beta$-system \eq{beta1}-\eq{beta2} becomes
\begin{align}
	\del_i \gamma^j &= Z_{ji}^j \gamma^j-Z_{jj}^i \gamma^i
	\qquad \mbox{for $1\leq i\neq j\leq n$,}\qquad \quad \big(\del_i=\textstyle\pd{}{w^i}\big)\label{beta1rich}\\
	Z_{ik}^j \gamma^j &= Z_{jk}^i \gamma^i
	\qquad\quad\qquad\,\,\mbox{for $1\leq k\neq i\,<j \,\neq k\leq n$,}\label{beta2rich}
\end{align}
where
\beq\label{pull_backs}
	\gamma^i(w):=\beta^i\circ\rho^{-1}(w)
	\quad\mbox{and} \quad Z_{ij}^k(w):=\Gamma_{ij}^k\circ\rho^{-1}(w)\,.
\eeq
Symmetry and flatness of the connection $\nabla$  imply  the following relations among 
the $Z_{ij}^k(w)$:
\beq\label{T0Z} 
	Z_{ij}^k= Z_{ji}^k\,,\qquad\forall\, i,\, j,\, k,\,\mbox{ (symmetry) }
\eeq
\beq\label{R0Z} 
	\del_m\big(Z^j_{ik}\big)-\del_k\big(Z^j_{im}\big) 
	=\sum_{t=1}^n\big(Z^j_{tk}Z^t_{im}-Z^j_{tm} Z^t_{ik}\big)\,,\qquad\forall\, i,\, j,\, k,\, m\,\mbox{ (flatness)}.
\eeq

\bigskip

\subsection{Solution of $\bet(\R)$ for systems with no algebraic part}\label{rich_no_constr}
We now assume that
\beq\label{no_alg}
	Z_{ij}^k=0 \qquad\text{whenever $\eps(i,j,k)=1$.}
\eeq
In this case $\beta(\R)$ reduces to the pure PDE system \eq{beta1rich}.
Darboux's theorem (see Theorem 4.1 in \cite{jk1}) applies and shows that the solutions to $\beta(\R)$ depend
on $n$ functions of 1 variable. More precisely:

\begin{theorem}\label{rich_unconstr} 
	Given a $C^2$-smooth, rich frame $\{r_1,\dots,r_n\}$ in a neighborhood 
	of $\br{w}\in\RR^n$. Let $(w^1,\dots,w^n)$ be associated Riemann 
	invariants and assume the normalization \eq{norm}.
	Let the connection coefficients $Z_{ij}^k$ be defined by \eq{pull_backs} and 
	assume that $Z^k_{ij}=0$ whenever $\eps(i,j,k)=1$. 
	
	Then, for given functions $\varphi_i$, $i=1,\dots,n$, of one variable, there is a unique local 
	solution $\gamma^1(w),\dots,\gamma^n(w)$ to the $\beta$-system \eq{beta1rich} with
	\[\gamma^i(\br{w}^1,\dots,\br{w}^{i-1}, w^i, \br{w}^{i+1}, \dots \br{w}^n)=\varphi_i(w^i).\]
\end{theorem}
\begin{proof}
The compatibility condition required by Darboux's theorem is that the mixed second order 
partial derivatives calculated from \eq{beta1rich} should agree. That is, we need to verify that
\beq\label{compat}
	\del_k\big(Z_{mj}^j \gamma^j-Z_{jj}^m \gamma^m\big) =
	\del_m\big(Z_{kj}^j \gamma^j-Z_{jj}^k \gamma^k\big) 
\eeq
holds whenever $\eps(j,k,m)=1$, when the derivatives are calculated according to \eq{beta1rich}.
Performing the differentiations, applying \eq{beta1rich}, and collecting terms, we obtain the conditions
\begin{align}
	\big(\del_m Z_{kj}^j-\del_k Z_{mj}^j \big)\gamma^j 
	&- \big(\del_m Z_{jj}^k + Z_{jj}^kZ_{mk}^k + Z_{jj}^mZ_{mm}^k -Z_{mj}^j Z_{jj}^k\big)\gamma^k\nn\\
	&+  \big(\del_k Z_{jj}^m + Z_{jj}^mZ_{km}^m + Z_{jj}^kZ_{kk}^m -Z_{kj}^j Z_{jj}^m\big)\gamma^m=0\,,
	\qquad\text{whenever $\eps(j,k,m)=1$.} \label{coeffc}
\end{align}
We shall show that  the coefficients of $\gamma^j$, $\gamma^k$, and $\gamma^m$ vanish identically 
due to flatness and symmetry. First, due to symmetry and to \eq{R0Z} with $i=j$, we have
\[\del_m Z_{kj}^j-\del_k Z_{mj}^j = \del_m Z_{jk}^j-\del_k Z_{jm}^j
=\sum_{t=1}^n \big(Z^j_{tk}Z^t_{jm}-Z^j_{tm} Z^t_{jk}\big)\,.\]
Recalling \eq{no_alg} and that $\eps(j,k,m)=1$ we conclude that the latter sum is zero. This shows 
that the coefficient of $\gamma^j$ in \eq{coeffc} vanishes identically. 

Next, using \eq{R0Z} with $i=k$ and then interchanging $j$ and $k$ yields the general identity
\[\del_m Z_{jj}^k -\del_j Z_{jm}^k=\sum_{t=1}^n \big(Z^k_{tj}Z^t_{jm}-Z^k_{tm} Z^t_{jj}\big)\,.\]
Then apply this to the case when $\eps(j,k,m)=1$ to get that
\beq\label{coeffc1}
	0=\del_m Z_{jj}^k - \sum_{t=1}^n \big(Z^k_{tj}Z^t_{jm}-Z^k_{tm} Z^t_{jj}\big)
	= \del_m Z_{jj}^k -Z_{jj}^k Z_{jm}^j +Z_{km}^k Z_{jj}^k +Z_{mm}^k Z_{jj}^m\,.
\eeq
Applying symmetry \eq{T0Z} in the last expression shows that the coefficient of $\gamma^k$ 
in \eq{coeffc} vanishes identically.

Finally, since we need to verify that the coefficient of $\gamma^m$ vanishes identically for all triples
$j,k,m$ with $\eps(j,k,m)=1$, we may as well interchange $m$ and $k$ in the coefficient of 
$\gamma^m$ in \eq{coeffc}. The result is the right-hand side in \eq{coeffc1}, which vanishes.
\end{proof}

In Example~\ref{rich_rank0} we consider the rich, orthogonal frame whose Riemann coordinates 
are cylindrical coordinates in $\RR^3$. According to the analysis in Section 4.2.1 of \cite{jk1}, the 
algebraic part of the $\lam$-system corresponding to any rich, orthogonal frame is necessarily trivial. 
By Observation~\ref{zero-rank} above the same is true for the algebraic part of the $\bet$-system.
This example illustrates Theorem~\ref{rich_unconstr}, as well as Observations~\ref{orthogonal} 
and~\ref{scaling}.

\begin{remark}
	We note that this result applies to any frame when $n=2$ since in that case
	there are no algebraic constraints. Theorem \ref{rich_unconstr} is well-known in the setting
	of a given, strictly hyperbolic system \eq{cons}, see \cites{cl,ser2,ts1}.
\end{remark}


\section{Analysis of $\beta(\mathfrak R)$ for $n=3$}\label{n=3}
We now restrict attention to frames on open subsets of $\RR^3$. Before considering the 
associated $\beta$-systems we recall the breakdown of possible cases for the 
$\lam$-systems. In \cite{jk1}, it was demonstrated that there are 
essentially only four possibilities in this case, as described in the following proposition:  
\begin{proposition}[\cite{jk1}]\label{n=3lambda}
Given a smooth frame $\R$ on $\Omega\subset\RR^3$, the solution set of $\lambda(\mathfrak R)$
is described by one of the following cases:
\begin{itemize}
	\item [I:] if $\rank[\eq{sev2}]=0$ (no algebraic constraints) then $\R$ is necessarily rich and 
	a general solution of $\lambda(\mathfrak R)$ depends on 3 
	functions of 1 variable. There are strictly hyperbolic solutions in this class. 
	\item [II:] if $\rank[\eq{sev2}]=1$ (a single algebraic constraint) then there are two possibilities:
	\begin{itemize}
		\item[IIa.] All three $\lam^i$ appear in the algebraic constraint.
	 	The solution of  the $\lam$-system is either trivial or depends on 2 arbitrary constants. 
		In the latter case, there are strictly hyperbolic systems in the family.
	 	There are no rich systems in class IIa.
		\item[IIb.] Exactly two $\lam^i$ appear in the algebraic constraint.
		Two $\lam^i$ coincide and the general solution is either trivial or depends on 1 arbitrary 
		function of 1 variables and 1 constant. There are no strictly hyperbolic systems, 
		but there are rich systems, in class IIb.
	\end{itemize}
	\item [III:] if $\rank[\eq{sev2}]=2$ then there are only trivial solutions 
	$\lam^1=\lam^2=\lam^3\equiv$ constant. 
\end{itemize}
Furthermore, each of the above possibilities can occur.
\end{proposition}

%
In our analysis of the $\beta$-system, Case I in Proposition \ref{n=3lambda} 
is covered by the analysis in Section \ref{rich_no_constr}.
We consider Case II in sections \ref{rich_rank=1_n=3} and \ref{non-rich_rank=1_n=3} below; 
in particular this covers the case of rich frames with one algebraic constraint. 
The various sub-cases in this latter category indicate 
that the analysis of rich systems of size $n\geq 4$ with algebraic constraints (i.e.\ the rich 
systems {\em not} treated in Section \ref{rich_no_constr}), is quite involved. 
Finally, Case III is largely irrelevant as it admits only trivial 
linear systems $u_t+\bar\lam u_x=0$, for which any scalar field is an 
extension. However, it does highlight the fact that our approach of characterizing 
extensions that satisfy \eq{cond4} due to orthogonality alone, may not provide
{\em all} extensions for {\em all} systems with a given eigen-frame $\R$; 
see Example \ref{strict_hyp_not_all}.

\subsection{The algebraic parts of $\lambda(\mathfrak R)$ and $\beta(\mathfrak R)$ for $n=3$}
\label{alg_n3}
Before analyzing $\beta(\mathfrak R)$ we establish a useful relationship between the 
algebraic parts of $\lambda(\mathfrak R)$ and $\beta(\mathfrak R)$. 
The former is given by \eq{sev2}:
\beq\label{Al}
	A_\lambda\, (\lam^1,\lam^2,\lam^3)^T=0\,, \qquad\text{where}\qquad A_\lam=
	\left[\begin{array}{ccc}
	c_{23}^1 &  \Gamma_{32}^1 & -\Gamma_{23}^1 \\
  	\Gamma_{31}^2&  c_{13}^2 &    - \Gamma_{13}^2\\
    	\Gamma_{21}^3&-\Gamma_{12}^3   &c_{12}^3   
	\end{array}\right].
 \eeq
For $n=3$ the algebraic part \eq{beta2} of the $\beta$-system is
\beq\label{Ab}
	A_\bet\, (\bet^1,\bet^2,\bet^3)^T=0\qquad\text{where}\qquad 
	A_\bet=\left[
	\begin{array}{ccc}
	c_{23}^1 &  -\Gamma_{31}^2 & \Gamma_{21}^3 \\
	- \Gamma_{32}^1&  c_{13}^2 &     \Gamma_{12}^3\\
	- \Gamma_{23}^1&\Gamma_{13}^2   &c_{12}^3   
	\end{array}\right].
 \eeq
A calculation shows that
\beq\label{Ab-Al} 
	A_\lam=\left[
	\begin{array}{ccc}
	  1&0   & 0  \\
	  0&-1   &0   \\
	  0& 0  &1   
	\end{array}
	\right]\, A_\beta^T\,\left[\begin{array}{ccc}
	  1&0   & 0  \\
	  0&-1   &0   \\
	  0& 0  &1   
	\end{array}
	\right]\,.
\eeq
Observing that $\rank(A_\lambda)\leq 2$, we have:
\begin{proposition}\label{n=3_ranks} 
	For $n=3$ the ranks of the algebraic parts of $\lambda(\mathfrak R)$ 
	and $\beta(\mathfrak R)$ are equal and $\leq 2$.
\end{proposition}
\begin{remark}\label{n=4_ranks}
	Example \ref{unequal_rank} shows that this result does not generalize to $n\geq 4$.
\end{remark}
%

\subsection{Rich frames with algebraic part of $\bet(\R)$ of rank 1}\label{rich_rank=1_n=3}
Given a rich frame $\mathfrak R=\{r_1,\,r_2,\,r_3\}$, we make a choice $w=(w^1,w^2,w^3)$ 
of Riemann coordinates  and scale the frame according to \eq{norm}. The matrices $A_\lam$ and 
$A_\beta$ are then 
\beq\label{Al-Ab-rich}
	A_\lam=\left[\begin{array}{ccc}
	0 &  Z_{23}^1 & -Z_{23}^1 \\
	  Z_{13}^2&  0 &    - Z_{13}^2\\
	    Z_{12}^3&-Z_{12}^3   &0   
	\end{array}\right]
	\qquad\text{and}\qquad 
	A_\bet=\left[\begin{array}{ccc}
	0 &  -Z_{13}^2 & Z_{12}^3 \\
	 - Z_{23}^1&  0 &     Z_{12}^3\\
	   - Z_{23}^1&Z_{13}^2   &0   
	\end{array}\right],
\eeq
where $Z_{ij}^k$ are given by \eq{pull_backs}${}_2$. Assume now that $\rank\eq{beta2rich}=1$,
i.e.\ $\rank(A_\lam)=\rank(A_\beta)=1$, and observe that this is the case if and only if exactly two 
of the functions $Z_{12}^3$, $Z_{13}^2$ and $Z_{23}^1$ vanish identically. 
Without loss of generality we assume that  
\beq\label{assump1}
	Z_{23}^1\neq 0,\qquad\text{while}\qquad Z_{12}^3=Z_{13}^2\equiv 0\,.
\eeq
\begin{remark}
	In this case the algebraic part \eq{sev2} of the $\lam$-system requires 
	$\lam^1\equiv \lam^2$ and, according to Proposition~\ref{n=3lambda} 
	there are only two possibilities for the solutions of the $\lam$-system: 
\begin{itemize} 
	\item only the trivial solution: 
	$\lam^1(u)=\lam^2(u)=\lam^3(u)\equiv \hat \lam\in\RR$ 
	(Example~\ref{rich_rank1_details});
	\item $\lam^1(u)=\lam^2(u)\neq \lam^3(u)$ and the general solution 
	depends on 1 arbitrary function of one variable and one arbitrary constant, 
	which in some examples may be absorbed in the arbitrary function. 
	(See Examples~\ref{rich_rank1_details}).
\end{itemize}
\end{remark}
We proceed to show that for ``rich, rank 1" frames, 
there are exactly three possibilities for the solution set of the  $\bet$-system. 
In all cases the corresponding extensions are degenerate. 
\begin{theorem}\label{rich_rank_1_thm}
	Given a rich frame $\R=\{r_1,\,r_2,\,r_3\}$ on $\Omega\subset\RR^3$,
	consider the $\beta$-system \eq{beta1rich}-\eq{beta2rich} expressed in 
	Riemann coordinates. Assume that the algebraic part \eq{beta2rich} of 
	$\beta(\R)$ has rank 1. Then there are three possibilities for the solution 
	set of the $\beta$-system:
	 \begin{itemize}
		\item[(1)] Only the trivial solution: $\bet^1=\bet^2=\bet^3\equiv 0$
		\item[(2)] Exactly two $\beta^i$ are zero and the third depends on 1 arbitrary 
		function of 1 variable.
		\item[(3)] Exactly one $\beta^i$ is zero and the other two 
		$\beta^i$ depend on 2 arbitrary functions of 1 variable.
	 \end{itemize}
\end{theorem}
\begin{proof}
As in Section \ref{rich}, we let $\gamma^i(w(u))=\beta^i(u)$ and  
$Z_{ij}^k(w(u))=\Gamma_{ij}^k(u)$.  Then the algebraic part \eq{beta2rich} of the 
$\bet$-system is  equivalent to
\beq\label{triv_bet1}
	\gamma^1\equiv0\,,
\eeq
showing that there are no non-degenerate extensions in this case.
The differential part \eq{beta1rich} of the $\bet$-system reduces to:
\begin{align}
	\label{IIBrich1} 0 &= \gamma^2\,Z_{11}^2\\
	\label{IIBrich2} 0&= \gamma^3\,Z_{11}^3\\
	\label {IIBrich3}	\del_1(\gamma^2) &= \gamma^2\,Z_{12}^2\\
	\label {IIBrich4}	\del_3(\gamma^2) &= \gamma^2\,Z_{32}^2 -\gamma^3\, Z_{22}^3\\
	\label {IIBrich5}	\del_1(\gamma^3) &= \gamma^3\,Z_{13}^3\\
	\label {IIBrich6} \del_2(\gamma^3) &= \gamma^3\,Z_{23}^3 -\gamma^2\, Z_{33}^2\,.
\end{align}
We proceed to consider the various possibilities:
\begin{enumerate}
\item[\bf a.] 
If $Z_{11}^2\neq 0$ and $Z_{11}^3\neq 0$ then $\gamma^2\equiv 0$ and 
$\gamma^3\equiv 0$. Together with \eq{triv_bet1} this shows that the $\bet$-system 
has only the trivial solution $\beta^1=\beta^2=\beta^3\equiv 0$ in this case.
\item[\bf b.] 
If  $Z_{11}^2\neq 0$ and $Z_{11}^3\equiv 0$ then from \eq{IIBrich1} it follows that 
$\gamma^2\equiv 0$,  and the system \eq{IIBrich1}-\eq{IIBrich6} reduces to:
\begin{align}
	\label {IIBrich4-b} 0&= \gamma^3\, Z_{22}^3\\
	\label {IIBrich5-b} \del_1(\gamma^3) &= \gamma^3\,Z_{13}^3\\
	\label {IIBrich6-b} \del_2(\gamma^3) &= \gamma^3\,Z_{23}^3.
\end{align}
If $Z_{22}^3\neq 0$ then \eq{IIBrich4-b} implies $\gamma^3 =0$ and the 
$\beta$-system has only the trivial solution $\beta^1=\beta^2=\beta^3\equiv 0$. 
Otherwise the system reduces to \eq{IIBrich5-b}-\eq{IIBrich6-b}, which we claim is
a compatible system. The only integrability condition is $\del_1(\del_2 \gamma^3)
=\del_2(\del_1 \gamma^3)$ which, upon using \eq{IIBrich5-b}-\eq{IIBrich6-b}, 
amounts to the condition that $\del_2 Z_{13}^3=\del_1 Z_{23}^3$. 
This is indeed the case by flatness \eq{R0Z}, symmetry \eq{T0Z}, and the 
assumptions \eq{assump1} and $Z_{11}^3\equiv 0$. We may therefore apply Darboux's 
theorem (Theorem 4.1 in \cite{jk1}) and conclude that $\gamma^3$ depends on 
1 arbitrary function of 1 variable.  
\item[\bf c.] 
If  $Z_{11}^2= 0$ and $Z_{11}^3\neq 0$ then from \eq{IIBrich2} it follows that 
$\gamma^3\equiv 0$, while $\gamma^2$ depends on 1 arbitrary function of 1 variable. 
This case is equivalent to Case {\bf b} by permutation of the third and the 
second eigenvectors (recall the symmetry \eq{T0Z} for rich systems).
\item[\bf d.] 
If $Z_{11}^2\equiv  Z_{11}^3\equiv 0$, then the differential system reduces to 
\eq{IIBrich3}-\eq{IIBrich6} for the unknown functions $\gamma^2$ and $\gamma^3$.
The two integrability conditions are $\del_3(\del_1 \gamma^2)=\del_1(\del_3 \gamma^2)$
and $\del_2(\del_1 \gamma^3)=\del_1(\del_2 \gamma^3)$, and these are satisfied as 
identities thanks symmetry, flatness, and the current assumptions. We apply 
Darboux's theorem and conclude that $\gamma^2$ depends on 1 
function of 1 variable, as does $\gamma ^3$.
Example~\ref{rich_rank1_details} belongs to this category.
\end{enumerate}

\bigskip

\end{proof}
%
%
%
%
%

\subsection{Non-rich frames with algebraic part of $\bet(\R)$ of rank 1}\label{non-rich_rank=1_n=3}
The following theorem lists all possible degrees of freedom that the solutions of the 
$\bet$-system may enjoy  in this case.
\begin{theorem}\label{thm_non-rich_n=3_short}
Given a non-rich frame $\R=\{r_1,\,r_2,\,r_3\}$ on $\Omega\subset\RR^3$
for which the algebraic part \eq{beta2} of $\beta(\R)$ has rank 1. Then the 
solution set of the $\beta$-system \eq{beta1}-\eq{beta2} is covered by one
of the following possibilities, all of which can occur:
\begin{itemize}
	\item[\bf (1)] Only the trivial solution: $\bet^1=\bet^2=\bet^3\equiv0$
	\item[\bf (2)] Exactly two $\beta^i$ are zero and the third depends on 1 arbitrary 
	function of one  variable.
	\item[\bf (3)] Exactly one $\beta^i$ is zero and the other two  $\beta^i$ depend on
	 \begin{itemize}
	\item[\bf (3a)] 2 arbitrary functions of one variable.
	\item[\bf (3b)] 1 common  arbitrary constant.
        \end{itemize}
        \item[\bf (4)]  There is a non-degenrate solution (all $\beta^i$ are non-zero) which depends on
        \begin{itemize}
	\item[\bf (4a)]  1 arbitrary function of one variable and 1 arbitrary constant.
	\item[\bf (4b)] 2 arbitrary constants.
	\item[\bf (4c)] 1 arbitrary constant.
        \end{itemize}
 \end{itemize}
\end{theorem}
\begin{proof} In Section~\ref{proof} we show that there are no other possibilities 
for the solutions set of the $\bet$-system. On the other hand, the examples in 
Section~\ref{exs} confirm that every scenario from the above list is realizable: 
(Cases with Roman numerals refer to the cases listed in Proposition \ref{n=3lambda})
\begin{itemize}
\item{\bf Case (1)} is realized by Example~\ref{non-rich-rank1_(ia)}. The corresponding 
$\lam$-system belongs to Case IIb  of Proposition~\ref{n=3lambda} and has non-trivial 
solutions. 
\item{\bf Case (2)} is realized by Examples~\ref{non-rich-rank1_(iib)} 
and~\ref{non-rich-rank1_(iiib)}. In the former  example, the corresponding 
$\lam$-system belongs to {Case IIa} and has non-trivial solutions depending 
on 2 constants. There are strictly hyperbolic  conservative systems that 
correspond to this frame. In the latter example, the corresponding $\lam$-system 
belongs to Case IIb and has non-trivial solutions depending on 1 function of 
one variable and 1 constant. This frame does not admit strictly hyperbolic 
conservative systems. 
\item{\bf Case (3a)} is realized by Examples~\ref{non-rich-rank1_strict_hyp_case_3b} 
and ~\ref{ex_jk1_5.3},  
Example~\ref{deg-family} when $g\equiv 0$, as well as Example~ \ref{ex_jk1_5.4}.
In the first four examples the $\lam$-system is of Case IIb, 
whereas in the last example the $\lam$-system is of the Case IIa. 
The $\lam$-systems in Examples~\ref{non-rich-rank1_strict_hyp_case_3b},
~\ref{ex_jk1_5.3} and~\ref{deg-family} 
have non-trivial solutions, while in the other two examples  the corresponding 
$\lam$-system has only the trivial solution. 
\item{\bf Case (3b)} is realized by Example~\ref{deg-family} with $g\not\equiv 0$, 
and the corresponding $\lam$-system belongs to Case IIb and has non-trivial solutions. 
\item{\bf Case (4a)} is realized by Examples~\ref{gas_dyn} (non-rich gas dynamics) 
and~\ref{non-rich-rank1_(iid)}. In the former example the corresponding $\lam$-system 
belongs to Case IIa and admits strictly hyperbolic  conservative systems, whereas in 
the latter example the corresponding $\lam$-system belong to Case IIb and does 
not admits strictly hyperbolic  conservative systems.
 \item{\bf Case (4b)} is realized by Example~\ref{non-rich-rank1_(ib)}. 
 The corresponding $\lam$-system belongs to Case IIb and has non-trivial solutions. 
\item{\bf Case (4c)} is realized by Example~\ref{non-rich-rank1_(iid1)}. 
The corresponding $\lam$-system belongs to Case IIb and has non-trivial solutions.
\end{itemize}
\end{proof}
\begin{remark} We note  that cases (1)-(3a) listed in Theorem~\ref{thm_non-rich_n=3_short} 
include  all cases from Theorem~\ref{rich_rank_1_thm}. Thus {\em non-richness} of the 
frame opens for more possibilities (Cases 3b and 4a-c).

By combining Theorem~\ref{thm_non-rich_n=3_short} and 
Theorem \ref{rich_rank_1_thm} we see that if a frame with a non-trivial algebraic 
constraint admits a non-degenerate extension, then the frame must necessarily be non-rich.

On the other hand, by Theorem~\ref{thm_non-rich_n=3_short} and 
Theorem~\ref{rich_unconstr}, the largest set of extensions occurs in the algebraically 
unconstrained (and hence) rich case described in Section \ref{no-alg}. Indeed,
in this case (i.e.\ $\Gamma_{ij}^k\equiv 0$ whenever $\eps(i,j,k)=1$) 
the general solution of the $\bet$-system depends on 3 arbitrary functions of one variable, 
while for any algebraically constrained case the maximal 
degree of freedom is 2 arbitrary functions of one variable (Case (3a)).
\end{remark}

\section{Completion of the proof of  Theorem \ref{thm_non-rich_n=3_short}}\label{proof}
In this section, we show that the cases listed in Theorem \ref{thm_non-rich_n=3_short} 
exhaust all possibilities for solutions of $\beta(\R)$, whenever $\R$ is a non-rich frame on $\Omega\subset\RR^3$
whose $\bet$-system has algebraic rank 1. 

\subsection{Setup}
By assumption the matrix $A_\beta$ in \eq{Ab} has rank $1$.
Before showing how each of the possibilities arise we make a choice of indices. 
By assumption $\R$ is non-rich such that there a triple $(i,j,k)$ with $\eps(i,j,k)=1$  
and $c^i_{jk}=\Gamma^i_{jk}-\Gamma^i_{kj}\neq 0$.
If necessary we permute the prescribed eigenfields $r_1,\,r_2,\,r_3$ and assume\beq\label{non-rich} 
	c_{32}^1\neq 0\qquad\text{and}\qquad \Gamma_{32}^1\neq 0\,.
\eeq 
We define 
\[a(u)=\frac {\Gamma_{32}^1(u)}{c_{32}^1(u)}\neq 0\,.\]
Since $\rank(A_\beta)=1$ it follows that
\beq\label{a_c}
	c_{31}^2=a\,\Gamma_{31}^2 \qquad\text{and}\qquad \Gamma_{12}^3=a\,\Gamma_{21}^3.
\eeq
Using \eq{non-rich} we solve the first algebraic constraint in \eq{Ab} for $\beta^1$ to get
\beq\label{beta11} 
	\bet ^1=\alpha^2\,\bet^2+\alpha^3\,\bet^3\qquad\mbox{where}\qquad
	\alpha^2=- \frac{\Gamma_{31}^2}{ c_{32}^1} \mbox{ and } 
	\alpha^3= \frac{\Gamma_{21}^3}{ c_{32}^1}.
\eeq
For $n=3$ the differential part of the $\beta$-system consists of the following $6$ PDEs:
\begin{align}
	\label {r2b1}      r_2(\bet^1) &= \bet^1\,(\Gamma_{21}^1 +c_{21}^1)-\beta^2\, \Gamma_{11}^2\\
	\label {r3b1}      r_3(\bet^1) &= \bet^1\,(\Gamma_{31}^1 +c_{31}^1)-\beta^3\, \Gamma_{11}^3\\
	\label {r1b2}	r_1(\bet^2) &= \bet^2\,(\Gamma_{12}^2 +c_{12}^2)-\beta^1\, \Gamma_{22}^1\\
	\label {r3b2}	r_3(\bet^2) &= \bet^2\,(\Gamma_{32}^2 +c_{32}^2)-\beta^3\, \Gamma_{22}^3\\
	\label {r1b3}	r_1(\bet^3) &= \bet^3\,(\Gamma_{13}^3 +c_{13}^3)-\beta^1\, \Gamma_{33}^1\\
	\label {r2b3}      r_2(\bet^3) &= \bet^3\,(\Gamma_{23}^3 +c_{23}^3)-\beta^2\, \Gamma_{33}^2\,.
\end{align}
Substituting \eq{beta11} into \eq{r2b1}-\eq{r2b3}, and using \eq{r3b2} and \eq{r2b3}, yield the 
PDE system
\begin{align}
	\label{r2b2m}  \alpha^2\,  r_2(\bet^2) &= 
	\bet^2\,\left[\alpha^2 \,(\Gamma_{21}^1 +c_{21}^1)-
	r_2\,(\alpha^2)\,+\alpha^3\,\Gamma_{33}^2-\Gamma_{11}^2\right]\\ 
	\nn  &\quad+\beta^3\left[\alpha^3\,(\Gamma_{21}^1+c_{21}^1
	-\Gamma_{23}^3-c_{23}^3)-r_2(\alpha^3)\right]\\
	\label{r3b3m} \alpha^3\,r_3(\bet^3) &= 
	\bet^2\,\left[\alpha^2 \,(\Gamma_{31}^1 +c_{31}^1-
	\Gamma_{32}^2 -c_{32}^2)-r_3\,(\alpha^2)\right]\\
	\nn &\quad+\beta^3\left[\alpha^3\,(\Gamma_{31}^1
	+c_{31}^1)-r_3(\alpha^3)+\alpha^2\,\Gamma_{22}^3-\Gamma_{11}^3\right]\\
	\label{r1b2m}
	r_1(\bet^2) &= \bet^2\,(\Gamma_{12}^2 +c_{12}^2-\alpha^2\,\Gamma_{22}^1)
	-\alpha^3\,\bet^3\, \Gamma_{22}^1\\
	\label{r3b2m} r_3(\bet^2) &= 
	\bet^2\,(\Gamma_{32}^2 +c_{32}^2)-\beta^3\, \Gamma_{22}^3\\
	\label{r1b3m} r_1(\bet^3) &= 
	\bet^3\,(\Gamma_{13}^3 +c_{13}^3-\alpha^3\, \Gamma_{33}^1)-\alpha^2\, \bet^2 \Gamma_{33}^1\\
	\label{r2b3m} r_2(\bet^3) &= \bet^3\,(\Gamma_{23}^3 +c_{23}^3)-\beta^2\, \Gamma_{33}^2\,.
\end{align}
We proceed to analyze this system according to  a number of $\bet$'s involved in the algebraic constraint \eq{beta11}.

\bigskip

\subsection{All three $\bet^i$ appear in the unique algebraic constraint}\label{case(i)}
In this case we have $\alpha^2\neq 0$ and $\alpha^3\neq 0$ (equivalently 
$\Gamma_{31}^2\neq 0$ and $\Gamma_{21}^3\neq 0$), and we can rewrite  
\eq{r2b2m}-\eq{r2b3m} as
\beq\label{r_i_beta_s}
	r_i(\beta^s)=\phi^s_i(u)\beta^2 + \psi^s_i(u)\beta^3\,,\qquad i=1,\, 2,\, 3,\quad s=2,\, 3.
\eeq
This is a ``Frobenius type" system which prescribes the derivatives of 2 unknowns 
along three independent vector-fields in $\RR^3$. According to the Frobenius Theorem \cite{sp1} 
the associated integrability conditions require that 
\beq\label{integrab}
	[r_i,r_j](\beta^s)=\sum_{k=1}^nc^k_{ij}\, r_k(\beta^s)
	\qquad 1\leq i<j\leq 3,\quad s=2,\, 3,
\eeq
should hold as identities in $u$, $\beta^2$, $\beta^3$, when both sides are evaluated
by using \eq{r_i_beta_s}. A calculation shows that these
requirements amounts to 6 relations of the form 
\beq\label{int-cond-i} 
	\mathcal A_{ij}^s(u)\beta^2+\mathcal B_{ij}^s(u)\beta^3=0\,,\qquad 1\leq i<j\leq 3,\quad s=2,\, 3,
\eeq
where the $\mathcal A_{ij}^s$ and $\mathcal B_{ij}^s$ are given in terms of 
$\phi^s_i$, $\psi^s_i$, their derivatives, and the $c^k_{ij}$. We have the following possibilities

\begin{itemize}
\item If all $\mathcal A_{ij}^s$ and $\mathcal B_{ij}^s$ vanish identically, then the integrability 
conditions are satisfied for all $\bet^2$ and $\bet^3$, and by the Frobenius Theorem the 
general solution of \eq{r_i_beta_s} depends on two arbitrary constants. The remaining 
$\bet^1$ is expressed in terms  of $\bet^2$ and $\bet^3$ using  \eq{beta11}. The two arbitrary 
constants can be chosen  so that all $\bet^i$ are non-zero on an open subset  of $\Omega$. 
This scenario is listed in {\bf \em Case 4b} of Theorem~ \ref{thm_non-rich_n=3_short}. 

 \item \medskip If the system \eq{int-cond-i} is of rank one then  there are two possibilities
 \begin{itemize}
\item Conditions \eq{int-cond-i} are satisfied when  $\bet^2\equiv0$ and $\bet^3 $
is an arbitrary function. Substituting into 
\eq{r2b2m}-\eq{r2b3m} yields an over-determined algebraic-differential system for 
$\beta^3$ with 3 algebraic constraints and 3 PDEs prescribing $r_i(\beta^3)$,
$i=1,\, 2,\, 3$.  Provided the algebraic relations are satisfied and the appropriate 
integrability conditions are met for all values of $\bet^3$, Frobenius' Theorem 
again applies and shows that the general solution of the $\beta^3$-system depends 
on 1 arbitrary constant. From  \eq{beta11} it follows that $\bet^1=\alpha^3\,\bet^3$ 
depends on the same constant, and we obtain    {\bf \em Case 3b} of the Theorem.   
Otherwise, $\bet^3\equiv0$ is the only solution and therefore from  \eq{beta11} it 
follows that $\bet^1=0$.  This is {\bf \em Case 1} (trivial solution).
\item Conditions \eq{int-cond-i} are satisfied when  $\bet^3\equiv0$ and $\bet^2 $
is an arbitrary function. By the same argument as above we are either in 
{\bf \em Case 3b} or  {\bf \em Case 1} of the Theorem.
\item Otherwise, from \eq{int-cond-i} one can write   $\bet^2=\Phi(u)\, \bet^3$, 
where $\Phi(u)$ is a non-zero function. Substituting into 
\eq{r2b2m}-\eq{r2b3m} yields an over-determined algebraic-differential system for 
$\beta^3$ with 3 algebraic constraints and 3 PDEs prescribing $r_i(\beta^3)$,
$i=1,\, 2,\, 3$.  Provided the algebraic relations are satisfied and the appropriate 
integrability conditions are met for all values of $\bet^3$, Frobenius' Theorem 
implies  that the general solution of the $\beta^3$-system depends on 1 
arbitrary constant. Then  $\bet^2=\Phi(u)\, \bet^3$ depends on the same constant, 
as does $\bet^1$ due to  \eq{beta11}. This is  {\bf \em Case 4c} of the Theorem.  
Otherwise the PDEs for $\beta^3$ imply that $\bet^3\equiv 0$ and hence 
$\bet^2=\Phi(u)\, \bet^3\equiv 0$. From  \eq{beta11} it follows that $\bet^1\equiv 0$ 
as well and we are in {\bf \em Case 1} (trivial solution).
\end{itemize}

\item \medskip Finally, if  the system \eq{int-cond-i} is of rank 2, then  $\beta^2=\beta^3\equiv0$ 
is the only solution for  \eq{r_i_beta_s} which, together with \eq{beta11},
yield the trivial solution of $\beta(\R)$ ({\bf \em Case~1}).
\end{itemize}


\medskip

\subsection{Exactly two $\bet^i$ appear in the algebraic constraint}\label{case(ii)}
In this case exactly one of $\alpha^2$ and $\alpha^3$ is non-zero; without loss of 
generality we assume that 
\beq\label{B2}
	\alpha^2\equiv 0 \qquad\text{and}\qquad \alpha^3\neq 0
	\qquad\text{(equivalently $\Gamma_{31}^2\equiv 0$ and $\Gamma_{21}^3\neq 0$)}.
\eeq
From \eq{beta11} it then follows that 
\beq\label{rank1case2}
	\bet^1=\alpha^3\bet^3,
\eeq 
and the differential part \eq{r2b2m}-\eq{r2b3m} of the $\beta$-system reduces to
\begin{align}
	\label{eeq1} 0 &= \bet^2\,\left(\alpha^3\,\Gamma_{33}^2-\Gamma_{11}^2\right)
	+\beta^3\left[\alpha^3\,(\Gamma_{21}^1+c_{21}^1-\Gamma_{23}^3-c_{23}^3)-r_2(\alpha^3)\right]\\
	\label{eeq2} \alpha^3\,r_3(\bet^3) &= \beta^3\left[\alpha^3\,(\Gamma_{31}^1+c_{31}^1)
	-r_3(\alpha^3)-\Gamma_{11}^3\right]\\
	\label{eeq3} r_1(\bet^2) &= \bet^2\,(\Gamma_{12}^2 +c_{12}^2)-\alpha^3\,\bet^3\, \Gamma_{22}^1\\
	\label{eeq4} r_3(\bet^2) &= \bet^2\,(\Gamma_{32}^2 +c_{32}^2)-\beta^3\, \Gamma_{22}^3\\
	 \label {eeq5} r_1(\bet^3) &= \bet^3\,(\Gamma_{13}^3 +c_{13}^3-\alpha^3\, \Gamma_{33}^1)\\
	\label {eeq6} r_2(\bet^3) &= \bet^3\,(\Gamma_{23}^3 +c_{23}^3)-\beta^2\, \Gamma_{33}^2\,.
\end{align}

From \eq{a_c}$_{1}$ and \eq{B2}$_{1}$ we conclude that $c_{13}^2\equiv 0$. It follows that 
$r_1$ and $r_3$ may be scaled to obtain {\em commuting} vector fields $\tilde r_1$ and $\tilde r_3$.  
We can then choose a coordinate system $(v^1,v^2,v^3)=\rho(u^1,u^2,u^3)$ in $\RR^3$ such 
that $\tilde r_1$ and $\tilde r_3$ are {\em coordinate} vectors. 
(Since $\tilde r_1$ and $\tilde r_3$ are in involution, Frobenius' 
Theorem gives the existence of a coordinate system $(\tilde u^1,\tilde u^2,\tilde u^3)$
in which the level sets of $\tilde u^2$, say, are integral manifolds of the 2-distribution 
$\text{span}\{\tilde r_1,\, \tilde r_3\}$. Since  $\tilde r_1$ and $\tilde r_3$ commute there is a change 
of the two coordinate functions $v^1=\tau^1(\tilde u^1,\tilde u^3)$,  $v^3=\tau^3(\tilde u^1,\tilde u^3)$ 
such that $\tilde r_1=\pd{}{v^1}$ and $\tilde r_3=\pd{}{v^3}$. Let $v^2=\tilde u^2$.)

In $v$-coordinates we then have 
\[r_1=\xi^1(v)\,\pd{}{v^1}\,,\quad r_2=\xi^2(v)\,\pd{}{v^1}+\xi^3(v)\,\pd{}{v^2}+\xi^4(v)\,\pd{}{v^3}\,,\quad 
r_3=\xi^5(v)\,\pd{}{v^3}\,,\] 
where $\xi^1$, $\xi^3$ and $\xi^5$ are non zero functions. For $i=1,\, 2,\, 3$, let $\gamma^i$ denote 
the pull-back of $\beta^i$ under the change of coordinates $\rho$: $\gamma^i(\rho(u))=\beta^i(u)$. 
We rewrite the system \eq{eeq1} -\eq{eeq6} in terms of $v$-coordinates and reorder the equations 
to group derivations of $\gamma^2$ and $\gamma^3$ together. With $\del_i=\pd{}{v^i}$ we then 
have the system
\begin{align}
\label{Valg} 0&= \mathcal A_1\, \gamma^2+ \mathcal B_1\, \gamma^3\\
\label{V12} \xi^1\,\del_1(\gamma^2) &= \Phi_{1}^2\,\gamma^2+\Psi_1^2\,\gamma^3\\
\label{V32} \xi^5\,\del_3(\gamma^2) &= \Phi_{3}^2\,\gamma^2+\Psi_3^2\,\gamma^3\\
 \label {V13} \xi^1\,\del_1(\gamma^3) &= \phantom{\Phi_{1}^3\,\gamma^2+}\Psi_1^3\,\gamma^3\\
\label {V23} [\xi^2\,\del_1+\xi^3\,\del_2+\xi^4\,\del_3](\gamma^3) &= \Phi_{2}^3\,\gamma^2+\Psi_2^3\,\gamma^3\\
\label{V33} \xi^5\,\del_3(\gamma^3) &=  \phantom{ \Phi_{3}^3\,\gamma^2+}\Psi_3^3\,\gamma^3\,,
\end{align}
where $\mathcal A$, $\mathcal B$, $\Phi^i_j$ and $\Psi^i_j$ are functions of $(v^1,v^2,v^3)$ that can 
be expressed in terms of functions $\Gamma^i_{jk}$, $\alpha^3$ and the change of coordinates $\rho$.

We observe that we can solve equations \eq{V12}-\eq{V33} for the derivatives of $\gamma^2$ and 
$\gamma^3$ and obtain a system of Darboux type, see \cite{jk1}. Derivatives of $\gamma^2$ are 
prescribed along two coordinate directions, while derivatives of $\gamma^3$ are prescribed along 
all three coordinate directions:
\begin{align}
        \label{v12} \del_1(\gamma^2) &= \phi_{1}^2\,\gamma^2+\psi_1^2\,\gamma^3\\
	\label{v32} \del_3(\gamma^2) &= \phi_{3}^2\,\gamma^2+\psi_3^2\,\gamma^3\\
	\label{v13} \del_1(\gamma^3) &= \phantom{\phi_{1}^3\,\gamma^2+}\psi_1^3\,\gamma^3\\
	\label{v23} \del_2(\gamma^3) &= \phi_{2}^3\,\gamma^2+\psi_2^3\,\gamma^3\\
	\label{v33} \del_3(\gamma^3) &=  \phantom{ \phi_{3}^3\,\gamma^2+}\psi_3^3\,\gamma^3\,,
\end{align}
where $\phi^i_j$ and $\psi^i_j$ are known functions of $(v^1,v^2,v^3)$. According to 
Darboux's theorem (Theorem 4.1 in \cite{jk1}) the relevant integrability conditions 
are obtained by evaluating the identities  $\del_i\del_j\gamma^s-\del_j\del_i\gamma^s=0$  
by using equations \eq{v12}-\eq{v33}. A calculation shows that this gives rise to 5 linear 
homogeneous algebraic equations on $\gamma^2$ and $\gamma^3$:
\beq\label{integrabi}
	 \mathcal A_i\,\gamma^2+  \mathcal B_i\,\gamma^3=0, \quad i=1,\dots,5, 
\eeq
where  $\mathcal A_i$ and  $\mathcal B_i$ are known functions of $(v^1,v^2,v^3)$.
Together with \eq{Valg} we therefore obtain a system $\mathcal L$ of 6 linear 
homogeneous algebraic equations on $\gamma^2$ and $\gamma^3$. We break down all 
possibilities according to the rank of $\mathcal L$.

\begin{itemize}
\item If $\rank\mathcal L=0$, i.e.\ the integrability conditions \eq{integrabi} together with \eq{Valg} 
are satisfied for all $\gamma^2$ and $\gamma^3$, then according to Darboux's theorem 
the general solution of  \eq{v12}-\eq{v33} depends on 1 arbitrary function
of one variable and 1 constant. This is {\bf \em Case 4a} of the theorem.
\item \medskip We split the case of $\rank\mathcal L=1$ into three sub-cases:
\begin{itemize}
\item $\mathcal L$ is satisfied for $\gamma^2\equiv 0$  (and so  $\beta^2\equiv 0$) and 
arbitrary $\gamma^3$. In this case system \eq{v12}-\eq{v33} reduces to
	\begin{align}
        \label{v12b2} 0 &= \psi_1^2\,\gamma^3\\
	\label{v32b2} 0 &= \psi_3^2\,\gamma^3\\
	\label{v13b2} \del_1(\gamma^3) &=\psi_1^3\,\gamma^3\\
	\label{v23b2} \del_2(\gamma^3) &= \psi_2^3\,\gamma^3\\
	\label{v33b2} \del_3(\gamma^3) &= \psi_3^3\,\gamma^3\,,
\end{align}
	If either $\psi_1^2\neq 0$ (equivalently $\Gamma_{22}^1\neq 0$) or  $\psi_3^2\neq 0$ 
	(equivalently $ \Gamma_{22}^3\neq 0$) then \eq{v12b2}, respectively \eq{v32b2}, 
	shows that $\gamma^3=0$. Thus $\bet ^3=0$ and $\bet^1=0$ by \eq{rank1case2}. 
	In this case the $\beta$-system has only the trivial solution ({\bf \em Case 1}). 
	Otherwise, \eq{v13b2}, \eq{v23b2}, \eq{v33b2} is a compatible system of Frobenius type 
	for $\gamma^3$. (One shows this by verifying that compatibility is a consequence of the 
	assumption that $\mathcal L$ is satisfied for $\gamma^2=0$  and arbitrary $\gamma^3$.
	We drop details of this routine calculation.) Therefore $\gamma^3$, and hence also 
	$\beta^3$, depends on 1 constant; $\beta^1=\alpha^3\bet^3$ (recall \eq{rank1case2})
	depends on the same constant, while $\bet^2=0$. This is {\bf \em Case 3b} of the theorem.
\item $\mathcal L$ is satisfied for $\gamma^3\equiv 0$  (and so $\beta^3\equiv 0$) and arbitrary 
	$\gamma^2$. Then the system \eq{v12}-\eq{v33} reduces to
	\begin{align}
        \label{v12b3} \del_1(\gamma^2) &= \phi_{1}^2\,\gamma^2\\
	\label{v32b3} \del_3(\gamma^2) &= \phi_{3}^2\,\gamma^2\\
	\label{v23b3} 0 &= \phi_{2}^3\,\gamma^2
\end{align}
If $\phi_2^3\not\equiv 0$ (equivalently $\Gamma_{33}^2\not\equiv 0$) then \eq{v23b3} shows
that $\gamma^2\equiv 0$, such that $\bet ^2\equiv 0$. Since $\beta^3\equiv 0$ also 
$\bet^1\equiv 0$ by \eq{rank1case2}, and we are in {\bf \em Case 1}. Otherwise, 
\eq{v12b3}-\eq{v32b3} is a compatible system of Darboux type for $\gamma^2$. 
(Again, compatibility is a consequence of our assumption that $\mathcal L$ is satisfied for 
$\gamma^3\equiv 0$  and arbitrary $\gamma^2$.) Therefore  $\gamma^2$  (and hence  $\beta^2$)  
depends on 1 arbitrary function of one variable and $\beta^1=\alpha^3\bet^3\equiv 0$. This is  
{\bf \em Case 2} of the theorem.

\item $\mathcal L$ is satisfied for $\gamma^3=\mathcal A(v)\gamma^2$, where  
$\mathcal A$ is a non-zero function. Substituting  $\gamma^3=\mathcal A\,\gamma^2$ into  
\eq{v12}-\eq{v33}, we obtain an overdetermined system of equations of Frobenious type 
(all partial derivatives of $\gamma^2$ are specified). If the system is compatible for all 
$\gamma^2$,  then by Frobenious theorem its general solution for $\gamma^2$ (and hence $\bet^2$) 
depends on 1 constant. Then $\gamma^3=\mathcal A\,\gamma^2$, and thus $\bet^3$, depends on 
the same constant. Combining this with \eq{rank1case2} we conclude that the the general solution 
of the $\bet$-system depends on 1 arbitrary constant. This is {\bf\em  Case 4c} of the theorem.
Otherwise, the system is  compatible only  for $\gamma^2\equiv 0$, and 
$\gamma^3=\mathcal A\,\gamma^2\equiv 0$. Hence $\bet^2=\bet^3\equiv 0$ in this case. 
By \eq{rank1case2}, also $\bet^1\equiv 0$, and the $\bet$-system has only the trivial solution 
({\bf\em  Case 1}).

\end{itemize}

\item Finally if $\rank\mathcal L=2$ then $\mathcal L$ is satisfied only for 
$\gamma^2=\gamma^3\equiv 0$. Hence $\bet^2=\bet^3\equiv 0$. Again $\bet^1\equiv 0$ 
and the $\bet$-system has only the trivial solution ({\bf  \em Case 1}).
\end{itemize}

\subsection{Exactly one $\beta^i$ appears in the algebraic constraint}\label{case(iii)}
In this case $\alpha^2=\alpha^3\equiv 0$ (equivalently $\Gamma_{31}^2=\Gamma_{21}^3\equiv 0$)
and $\bet^1\equiv 0$ by \eq{beta11}. The system \eq{r2b2m}-\eq{r2b3m} thus reduces to 
\begin{align}
	\label{r2b2d} 0 &= \bet^2\,\Gamma_{11}^2\\
	\label{r3b3d} 0&= \beta^3\,\Gamma_{11}^3\\
	\label {r1b2d} r_1(\bet^2) &= \bet^2\,(\Gamma_{12}^2 +c_{12}^2)\\
	\label {r3b2d} r_3(\bet^2) &= \bet^2\,(\Gamma_{32}^2 +c_{32}^2)-\beta^3\, \Gamma_{22}^3\\
	\label {r1b3d} r_1(\bet^3) &= \bet^3\,(\Gamma_{13}^3 +c_{13}^3)\\
	\label {r2b3d} r_2(\bet^3) &= \bet^3\,(\Gamma_{23}^3 +c_{23}^3)-\beta^2\, \Gamma_{33}^2\,.
\end{align}
From the assumptions of this case, together with \eq{a_c} and \eq{beta11} we have 
\beq\label{asumpiii}
	\Gamma^2_{31}=\Gamma^2_{13}=
	c^2_{13}=\Gamma^3_{12}=\Gamma^3_{21}=c^3_{12}\equiv 0
\eeq
We consider the possible sub-cases:
\begin{itemize}
	\item $\Gamma_{11}^2\not\equiv 0$ and $\Gamma_{11}^3\not\equiv 0$. In this case 
	$\beta^2\equiv 0$ and $\beta^3\equiv 0$, such that the $\beta$-system has only the trivial solution,
	i.e.\ {\bf \em Case 1} of the theorem.

	\medskip

	\item $\Gamma_{11}^2\equiv 0$ and $\Gamma_{11}^3\not\equiv 0$. In this case $\beta^3\equiv 0$ and
	the system \eq{r2b2d}-\eq{r2b3d} reduces to
	\begin{align}
	\label{eq1} 0 &= \bet^2\,\Gamma_{33}^2\\
	\label {eq2} r_1(\bet^2) &= \bet^2\,(\Gamma_{12}^2 +c_{12}^2)\\
	\label {eq3} r_3(\bet^2) &= \bet^2\,(\Gamma_{32}^2 +c_{32}^2).
	\end{align}
	If $\Gamma_{33}^2\neq 0$ then the $\beta$-system has only the trivial solution ({\bf \em Case 1}). 
	Otherwise a direct calculation (using flatness  and symmetry conditions \eq{T0}, \eq{R0}, as well 
	as assumptions  \eq{asumpiii} and  $\Gamma^2_{11}=\Gamma^2_{33}\equiv 0$
	that are now in force), shows that the compatibility condition
	\[[r_1,r_3](\beta^2)=\sum_{k=1}^3c^k_{13}r_k(\beta^2)\]
	holds as an identity in $u$, $\beta^3$, when calculated according to \eq{eq2} and \eq{eq3}.
	Since $c_{13}^2\equiv 0$ we can introduce the same coordinates as in Section~\ref{case(ii)}. 
	Rewritten in these new coordinates, the system \eq{eq2}, \eq{eq3} specifies the derivatives 
	of one unknown function in two coordinate directions and   therefore is of Darboux type. 
	Applying Darboux's theorem (see \cite{jk1}) we conclude that $\bet^2$ depends on 1 function 
	of one variable. In this case $\bet^1=\bet^3\equiv0$ and we obtain {\bf \em Case 2} of the 
	theorem.	
	
	\medskip

	\item $\Gamma_{11}^3\equiv 0$ and $\Gamma_{11}^2\neq 0$. This case reduces to the 
	previous one upon permuting the second and third eigenvectors. We obtain either a trivial 
	solution, or $\bet^1=\bet^2\equiv 0$ while $\bet^3$ depends on 1 arbitrary function of one 
	variable. This is again {\bf \em Case 2} of the theorem. 
	\medskip

	\item $\Gamma_{11}^2\equiv  \Gamma_{11}^3\equiv 0$. In this case \eq{r2b2d} and 
	\eq{r3b3d} are satisfied for all $\bet^2$ and $\bet^3$.  A direct computation (using flatness 
	and symmetry conditions \eq{T0}, \eq{R0}, as well as assumptions \eq{asumpiii} and 
	$\Gamma^2_{11}=\Gamma^3_{11}\equiv 0$ that are now in force), shows that  the compatibility 
	conditions for the remaining PDE system \eq{r1b2d}-\eq{r2b3d}, viz.
	\[[r_1,r_3](\beta^2)=\sum_{k=1}^3c^k_{13}r_k(\beta^2),\qquad 
	[r_1,r_2](\beta^3)=\sum_{k=1}^3c^k_{12}r_k(\beta^3)\,,\]  
	hold as identities in $u$, $\beta^2$ and $\beta^3$, when calculated according to 
	\eq{r1b2d}-\eq{r2b3d}.
\begin{remark} Explicitly, the integrability conditions are given by 
\begin{align}
	\label{1beta2} r_3(\Gamma_{12}^2+c_{12}^2)-r_1(\Gamma_{32}^2+c_{32}^2)
	+c_{13}^1\,(\Gamma_{12}^2+c_{12}^2)+c_{13}^3\,(\Gamma_{32}^2+c_{32}^2) &=0\\
	\label {1beta3} r_1(\Gamma_{22}^3) -\Gamma_{22}^3\,(\Gamma_{12}^2
	+c_{12}^2-\Gamma_{13}^3)&= 0\\
	\label{2beta2} r_2(\Gamma_{13}^3+c_{13}^3)-r_1(\Gamma_{23}^3+c_{23}^3)
	+c_{12}^1\,(\Gamma_{13}^3+c_{13}^3)+c_{12}^2\,(\Gamma_{23}^3+c_{23}^3) &=0\\
	\label {2beta3} r_1(\Gamma_{33}^2) -\Gamma_{33}^2\,(\Gamma_{13}^3
	+c_{13}^3-\Gamma_{12}^2)&= 0\,.
\end{align}
\end{remark}
\noindent
Unfortunately no change of variables seems to bring \eq{r1b2d}-\eq{r2b3d} into Darboux 
form (i.e. we do not obtain an equivalent system in which each equation contains only 
one partial derivative of one unknown function). We therefore need apply the more general 
Cartan-K\"ahler theorem.  
We omit the lengthy calculations that lead to the conclusion that the general solution of 
\eq{r1b2d}-\eq{r2b3d} depends on 2 arbitrary functions of one variable and two constants. 
Two arbitrary constants specify the values of $\bet^2$ and $\bet^3$ at an initial point 
$\bar u$ and 2 arbitrary functions of one variable  prescribe the directional derivatives 
$s^2=r_2(\bet^2)$  and $s^3=r_3(\bet^3)$ along a curve. These arbitrary functions absorb 
the arbitrary constants. Thus, the general solution depends on 2 arbitrary functions of one 
variable. Recalling that $\bet^1\equiv 0$ we conclude that this provides {\bf \em Case 3a} 
of the theorem.
\end{itemize}

With this we have verified that Theorem \ref{thm_non-rich_n=3_short} describes all possible 
degrees of freedom for the general solution of $\bet$-system, when its algebraic part has 
rank 1.

\section{Examples}
\label{exs}
\subsection{The Euler system for 1-dimensional compressible flow}
\begin{example}\label{gas_dyn}	
Extensions and entropies for the 1-d compressible Euler system 
have been considered by several authors, (see Remark \ref{euler_refs}). We now treat 
this particular case within our setup of prescribed eigen-frames. I.e., we first determine 
the eigen-frame $\R$ of the Euler system and then analyze the associated $\lam$- and 
$\bet$-systems. In Lagrangian variables the system is:
\bea
	v_t - u_x &=& 0  \label{mass}\\
	u_t + p_x&=& 0  \label{mom}\\
	E_t + (up)_x &=& 0\,,\label{energy}
\eea
where $v,\, u,\, p$ are the specific volume, velocity, and pressure, respectively, 
and $E=\epsilon+\frac{u^2}{2}$ is the total specific energy. Here $\epsilon$ denotes 
the specific internal energy and we assume that it is given in the form of a so-called 
complete equation of state (EOS) $\eps=\eps(v,S)$, \cite{mp}.
The thermodynamic variables are related through Gibbs' relation $d\epsilon=TdS-pdv$,
where $T$ is absolute temperature and $S$ is specific entropy. We make the standard 
sign assumptions:
\[T=T(v,S)=\eps_S(v,S)>0\,,\qquad \text{and}\qquad p=p(v,S)=-\eps_v(v,S)>0\,.\]
Thermodynamic stability requires that $\epsilon$ is convex at each state $(v,S)$ \cite{mp}: 
\beq\label{thdyn_stab}
	\eps_{vv}> 0\,,\qquad \eps_{vv}\eps_{SS}>\eps_{vS}^2\,,
\eeq
which implies that
\beq\label{stnd_assump}
	p_v(v,S)=-\eps_{vv}(v,S)<0\,.
\eeq
For smooth solutions \eq{mass}-\eq{energy} may be rewritten as
\bea
	v_t - u_x &=& 0 \label{Cont}\\
	u_t + p_x&=& 0  \label{Mom}\\
	S_t &=& 0\,. \label{Ener}
\eea
The Jacobian of the flux $f(v,u,S)=(-u,\, p,\, 0)^T$ has eigenvalues 
\[\lambda^1=-\sqrt{-p_v}\,,\qquad 
\lambda^2\equiv 0\,,\qquad 
\lambda^3=\sqrt{-p_v}\,,\]
with corresponding right and left eigenvectors (normalized according to $R_i\cdot L^j\equiv \delta_i^j$)
\beq\label{r_evs}
	R_1=\left[\,1,\, \sqrt{-p_v},\, 0\,\right]^T\,,\qquad
	R_2=\left[\,-p_S,\, 0,\, p_v\,\right]^T\,,\qquad
	R_3=\left[\,1,\, -\sqrt{-p_v},\, 0\,\right]^T\,,
\eeq
and 
\beq\label{levs}
	L^1=\textstyle\frac{1}{2}\Big[\, 1\,,\, \frac{1}{\sqrt{-p_v}}\,,\, \frac{p_S}{p_v}\, \Big]\,,\qquad
	L^2=\Big[\, 0\,,\, 0\,,\, \frac{1}{p_v}\, \Big]\,,\qquad
	L^3=\textstyle\frac{1}{2}\Big[\, 1\,,\, -\frac{1}{\sqrt{-p_v}}\,,\, \frac{p_S}{p_v}\, \Big]\,.
\eeq
The $\lambda$-system associated with the frame $\R=\{R_1,\, R_2,\, R_3\}$ was analyzed in 
\cite{jk1}, and there are two distinct cases:
\begin{itemize}
	\item [$(a)$] $\big(\frac{p_S}{p_v}\big)_v\equiv 0$: The Euler system is rich 
	with no algebraic constraints, and the general solution of the $\lambda$-system 
	depends on three functions of one variable.
	\item [$(b)$] $\big(\frac{p_S}{p_v}\big)_v\neq 0$: There is a single 
	algebraic relation among the eigenvalues: 
	\[\lambda^1+\lambda^3=2\lambda^2\,.\]
	The general solution of the $\lambda$-system is described by {\bf \em Case IIa}  
	of Proposition~\ref{n=3lambda} and  depends on 2 constants 
	$\bar\lam$, $C$ according to:
	\[\lambda^1 = \bar \lambda-C\sqrt{-p_v}  \,,\qquad
	\lambda^2 \equiv \bar \lambda \,,\qquad
	\lambda^3 = \bar \lambda+C\sqrt{-p_v}\,. \]
\end{itemize}
We note that Case (a) occurs if and only if the pressure has the particular form
\beq\label{spec_p}
	p(v,S)= \mathcal P(v+\phi(S))
\eeq
for functions $\mathcal P(\cdot)$, $\phi(\cdot)$ of 1 variable.

We proceed to analyze the $\beta$-system corresponding to $\R$.
A calculation (using the expressions for the coefficients $\Gamma_{ij}^k$ from \cite{jk1})
shows that the three algebraic relations of $\beta(\mathfrak R)$ are identical and express the 
single constraint that
\beq\label{eul_alg_beta_constr}
	\Big(\frac{p_S}{p_v}\Big)_v\big(\beta^1-\beta^3\big)=0\,.
\eeq
The two cases (a) and (b) thus yield different answers:
\begin{itemize}
	\item [$(a)$] $\big(\frac{p_S}{p_v}\big)_v\equiv 0$: There is no algebraic constraint
	in the $\beta$-system, and the analysis in Section \ref{rich_no_constr}
	applies. The general solution of the $\beta$-system depends on 3 functions of 1 variable. 
	\item [$(b)$] $\big(\frac{p_S}{p_v}\big)_v\neq 0$: The unique algebraic constraint 
	is $\beta^1\equiv \beta^3$. We claim that the $\bet$-system in this case is described 
	by {\bf \em Case 4a} of Theorem \ref{thm_non-rich_n=3_short}: the general 
	solution of the $\beta$-system depends on 1 constant and 1 function of 1 variable.
	Furthermore, this variable is the physical entropy $S$. 
	Indeed, the solution of the $\bet$-system provided by our MAPLE code is
	\beq\label{gas_beta}
		\bet^1=\bet^3=K_1\, p_v \qquad\mbox{and}\qquad 
		\bet^2= \frac{K_1 p_v^2}{2}\left(\int_{K_2}^{v} 
		p_{SS}(\tau,S)\,d\tau-\,\frac{ p_S^2}{p_v}(v,S)+F(S)\right)\,,
	\eeq
	where $K_1$ and $K_2$  are arbitrary constants and $F$ is an arbitrary function of one variable. 
	Since $p=-\eps_v$ we may re-write this as 
	\beq\label{gas_beta2}
		\bet^1=\bet^3=K\, \eps_{vv} \qquad\mbox{and}\qquad 
		\bet^2= \frac{K \eps_{vv}^2}{2}\left[\frac{\eps_{vv}\eps_{SS}-\eps_{vS}^2}{\eps_{vv}}+G(S)\right]\,,
	\eeq
	where we have set $K:=-K_1$ and $G(S):=-\eps_{SS}(K_2,S)-F(S)$. 

\end{itemize}
Let's focus on the non-rich case and determine the extensions $\eta(v,u,S)$ for the Euler system. 
Using the expressions in \eq{r_evs} and \eq{gas_beta2} we need to determine $\eta$ from the 
six relations:
\begin{align*}
	R^T_{1}(D^2\eta) R_{1}&=K e_{vv}\\
	R^T_{3}(D^2\eta) R_{3}&=K e_{vv}\\
	R^T_2(D^2\eta) R_2&=
	\frac{K \eps_{vv}^2}{2}\left[\frac{\eps_{vv}\eps_{SS}-\eps_{vS}^2}{\eps_{vv}}+G(S)\right]\\
	R^T_{i}(D^2\eta) R_{j}&=0 \qquad\mbox{ for $1\leq i<j\leq 3$}\,.
\end{align*}
A straightforward calculation shows that these imply that $\eta(v,u,S)$ has the form:
\beq\label{eta}
	\eta(v,u,S)=C_1v+C_2u+C_3\Big[\eps(v,S)+\frac{u^2}{2}\Big]+H(S)\,,
\eeq
From \eq{gas_beta2} we can immediately determine when the extension is convex:
the scalar field $\eta(v,u,S)$ given by \eq{eta} is strictly convex if and only if 
\[C_3\eps_{vv}>0\qquad\text{and}\qquad \frac{C_3}{\eps_{vv}}\big(\eps_{vv}\eps_{SS}-\eps_{vS}^2\big)+H''(S)>0
\qquad\text{for all $(v,S)\in\RR_+\times\RR$.}\] 
Under the assumption of thermodynamic stability \eq{thdyn_stab}, $\eta$ is convex if and only if
\[C_3>0\qquad\text{and}\qquad H''(S)>-\frac{C_3}{\eps_{vv}}\big(\eps_{vv}\eps_{SS}-\eps_{vS}^2\big)\,.\]

\begin{remark}\label{euler_refs}
	Extensions and entropies for the Euler system with general equations of state
	have been analyzed in \cites{rj, hllm, ser1}. 
	In \cite{rj} the general form of an extension was derived for Case (b).
	In \cite{ser1} the general form of extensions (for both Case (a) and Case (b)) was 
	determined (\cite{ser1}  Exercise 3.19, p.\ 86), and  the convexity of extensions of the 
	form $g(S)$ ($g$ a scalar map, $S$ the physical entropy) was characterized
	(\cite{ser1}  Exercise 3.18, p.\ 85). The latter issue was also treated in \cite{hllm}, 
	extending the analysis in \cite{ha} for ideal, polytropic gases.
\end{remark} 

\end{example}

\subsection{Examples with rich frames}
\subsubsection{\bf{Rich frames admitting strictly hyperbolic conservative systems}}
According to the analysis in Section \ref{rich}, any $n$-frame with the property that 
$\Gamma_{ij}^k$ vanishes identically whenever $\epsilon(i,j,k)=1$, is necessarily rich. 
Frames of this type admit a large family 
of corresponding conservative systems \eq{cons}, parameterized by $n$ arbitrary functions of 
one variable. The family contains both strictly hyperbolic and non-strictly hyperbolic systems. 
The solutions of the $\bet$-system enjoy the same degree of freedom (Theorem \ref{rich_unconstr}). 
For a strictly hyperbolic systems the solutions to the $\bet$-system provides {\em all} possible extensions, 
while non-strictly hyperbolic system may have additional extensions arising from the first part of the 
condition \eq{cond4}. 
\begin{example}\label{strict_hyp_not_all}
	This example shows that 
	$\beta(\R)$ may not provide all extensions for all systems \eq{cons} corresponding to a given frame, 
	even when the latter admits strictly hyperbolic systems \eq{cons}.
	Let $\mathfrak R=\{\del_{u^1},\del_{u^2}\}$, the standard coordinate frame on $\RR^2$.
	In this case any pair of functions of the form $(\lam^1(u),\lam^2(u))
	=(\phi(u^1),\psi(u^2))$ solves $\lam(\mathfrak R)$, which thus 
	admits strictly hyperbolic solutions. From \eq{cond4}, and the fact that 
	the $\lam$- and $\beta$-systems are ``separated" (no unknown of one 
	occurs in the other), one might expect that the $\beta$-system in this 
	case would provide {\em all} extensions for {\em all} systems \eq{cons} with the 
	same eigen-frame $\mathfrak R$. However, a simple example shows that this 
	is incorrect. E.g., among the systems with eigen-frame $\mathfrak R$ 
	there are the trivial systems $u_t+\bar\lam u_x=0$ ($\bar\lam\in\RR$), 
	which admits any scalar function $\tilde \eta(u)$ as an extension. 
	Consider the choice $\tilde \eta(u)=(u^1u^2)^2/2$. The lengths of the given eigenvectors,
	measured with respect to the inner-product $D^2 \tilde \eta(u)$, are $\beta^1=(u^2)^2$ 
	and $\beta^2=(u^1)^2$. However, these do {\em not} solve $\beta(\mathfrak R)$,
	which in this case consists of the two PDEs $\del_{u^i}\beta^j=0$, $1\leq i\neq j\leq 2$.
\end{example}
\begin{example}\label{rich_rank0}({\sc rich orthogonal frame})
$$R_1:=(u^1,\, u^2,\, 0)^T,\qquad R_2=(-u^2,\, u^1,\, 0)^T,\qquad R_3=(0,\, 0,\, 1)^T.$$ 
This is a rich, orthogonal, and commutative frame on $\RR^3-\{(0,0,u_3)\}$.
The $\lam$- and $\bet$-systems impose no algebraic constraints and 
its general  solution depends on 
3 arbitrary functions of one variable (Theorem \ref{rich_unconstr}). 
Introducing $v=(u^1)^2+(u^2)^2$, we have
$$\lam^1=F_1(v), \qquad \frac 1{\sqrt{v}}\,\int_*^{\sqrt{v}} 
F_1(\tau^2)\,d\tau+\frac{1}{u^1}\,F_2\left(\frac {u^2}{u^1}\right), 
\qquad \lam^3 =F_3(u^3);$$
$$\displaystyle{\bet^1=v\,G_1(v), \qquad 
\bet^2 =\sqrt{v}\,\int_*^{\sqrt{v}} {G_1(\tau^2)}\,d\tau+u^1\,G_2\left(\frac {u^2}{u^1}\right), \qquad 
\bet^3 =G_3(u^3).}$$ 
The frame is orthogonal, but not orthonormal; in accordance with Observation \ref{orthogonal}
the solutions of the $\bet$-system may be obtained by scaling the solutions of the $\lam$-system. 
Indeed, choosing 
\[G_1\equiv F_1\,,\qquad G_2(\xi)=(1+\xi^2)F_2(\xi)\,,\qquad \text{and}\qquad G_3\equiv F_3\,,\]
we obtain
\[\bet^1=v\lam^1,\qquad \bet^2=v\lam^2,\qquad \bet^3=\lam^3\,.\]
We also note any solution of the 
$\lam$-system can be combined with any solution of the $\bet$-system. I.e., by choosing 
a particular set of functions $F_1,F_2,F_3$ we specify $\lam^1,\lam^2,\lam^3$, and hence 
a conservative system \eq{cons} (unique up to adding a trivial flux). For this fixed conservative 
system {\em any} choice of functions $G_1,G_2, G_3$ will provide us with an extension.     

The Riemann coordinates for this frame (in the first octant, say) are
\[w^1={\textstyle\frac 1 2} \ln\big[(u^1)^2+(u^2)^2\big]\,,\qquad
w^2 =\arctan\left(\frac{u^2}{u^1}\right)\,,\qquad w^3 = u^3\,.\]
In terms of the Riemann coordinates we have
\begin{align*}
	\bet^1&=e^{2w^1}\,G_1(e^{2w^1})=\psi_1(w^1)\\
	\bet^2 &= e^{w^1}\,\int_*^{e^{w^1}} {G_1(\tau^2)}\,d\tau
	+e^{w^1}\cos(w^2)\,G_2\left(\tan w^2\right)
	= e^{w^1}\,\int_*^{e^{w^1}} \frac {\psi_1(\ln \tau)}{\tau^2}\,d \tau+e^{w^1}\,\psi_2(w^2)\\
	\bet^3 &=G_3(w^3)=\psi_3(w^3)\,,
\end{align*}
where $\psi_1$, $\psi_2$ and $\psi_3$ are arbitrary functions of 1 variable.
In accordance with Theorem \ref{rich_unconstr}, for a fixed  point 
$(\bar w^1,\bar w^2,\bar w^3)$ and three arbitrary  functions $\varphi_1(w^1), \varphi_2(w^2)$ and 
$\varphi_3(w^3)$, there is a unique solution of the $\bet$-system such that
$$\bet^1( w^1, \bar w^2,\bar w^3)=\varphi_1(w^1),\qquad  \bet^2(\bar w^1, w^2,\bar w^3)
=\varphi_2(w^2)\qquad \mbox{ and }
\qquad\bet^3(\bar w^1, \bar w^2, w^3)=\varphi_3(w^3),$$
Indeed, $\psi_1$, $\psi_2$ and $\psi_3$ are uniquely determined from the equations:
$$\psi_1(w^1)= \varphi_1(w^1), \quad e^{\bar w^1}\,\int_*^{e^{\bar w^1}} 
\frac {\psi_1(\ln \tau))}{\tau^2}\,d \tau+e^{\bar w^1}\,\psi_2(w^2)=\varphi_2(w^2) \mbox{ and }
\psi_3(w^3)=\varphi_3(w^3).$$
Note that dependence of the general solution on the constant $\bar w^1$ 
can be ``hidden" in the arbitrary 
functions by replacing arbitrary function $\varphi_2$ with  $e^{-\bar w^1}\, 
\varphi_2-\int_*^{e^{\bar w^1}} \frac {\varphi_1(\ln \tau))}{\tau^2}\,d \tau$. 
\end{example}

\subsubsection{\bf{Rich frames  with no corresponding strictly 
hyperbolic conservative systems}}\label{rich-const}

There are rich frames for which there exists a triple $i,j,k$, such that 
$\epsilon(i,j,k)=1$ and $\Gamma_{ij}^k\neq0$. Such frames do not 
admit strictly hyperbolic  conservative systems and in some cases 
admit only trivial  systems. Even in the latter case  the $\bet$-system 
may have non-trivial solutions. This is another indication of the lack of 
general relationships between the number of solutions to $\lam(\R)$ 
and $\bet(\R)$.

\begin{example}\label{rich-rank2} {\sc Rich frame with only trivial solutions 
of the $\lam$-system and non-trivial, but degenerate solutions of the $\bet$-system:}
\[R_1=(u^1,\,u^2,\, 0)^T,\qquad R_2=(-u^2,\, u^1,\, 0)^T, \qquad R_3=(-u^2,\, u^1,\, 1)^T.\]
This is a rich frame on $\RR^3-\{(0,0,u_3)\}$ with $\rank \eq{beta2}=\rank \eq{sev2}=2$. 
Hence the $\lam$-system has only the trivial solution $\lam^1=\lam^2=\lam^3=C$. 
The algebraic constraints of the $\bet$-system are $\bet^1=0=\bet^1+\bet^2$, and 
the general solution of the $\bet$-system depends on 1 arbitrary function of one variable:
\[\bet^1=\bet^2 =0, \qquad \bet^3 =F(u^3).\]
The corresponding extensions (modulo affine parts) are given by
\[\eta(u^1,u^2,u^3) =G(u^3)\,,\qquad\text{where $G''=F$.}\]
\end{example}
\begin{example}\label{rich_rank1_details} {\sc Rich frame with no strictly-hyperbolic  
solutions of the $\lam$-system and non-trivial, but degenerate solutions of the $\bet$-system:}
$$R_1:=(u^1,\, u^2,\, u^3)^T,\qquad R_2=(u^1,\, u^2,\, 0)^T,\qquad R_3=(u^1,\, 0,\, u^3)^T.$$ 
This is a rich {\em commutative} frame on  $\RR^3\setminus\{(0,0,u_3)\}$. The $\bet$-system 
contains 1 algebraic constraint and its solution set is described by part (3) of Theorem 
\ref{rich_rank_1_thm}: the general solution depends on 2 arbitrary functions $G_1$ and 
$G_2$ of one variable,
$$\bet^1=0, \qquad \bet^2 =u^1\,G_1\Big(\frac {u^3}{u^1}\Big), 
\qquad \bet^3 =u^1\,G_2\Big(\frac {u^2}{u^1}\Big).$$ 
A set of Riemann coordinates for this frame (in the first octant, say) are
\[w^1 =\ln u^2+\ln u^3-\ln u^1\,,\qquad w^2 =\ln u^1- \ln u^3\,,\qquad w^3 =\ln u^1- \ln u^2\,.\]
In terms of the Riemann coordinates we have
\[\bet^1(w) = 0\,,\qquad
\bet^2(w) = e^{w^1+w^3}\,\psi_2(w^2)\,,\qquad
\bet^3 (w)=e^{w^1+w^2}\,\psi_3(w^3)\,,\]
where $\psi_2$, $\psi_3$  are arbitrary functions. In accordance with Theorem 
\ref{rich_unconstr}, for a fixed point $(\bar w^1,\bar w^2,\bar w^3)$ and 
two arbitrary  functions $\varphi_2(w^2)$ and $\varphi_3(w^3)$, there is a unique 
solution of the $\bet$-system with
$$\bet^2(\bar w^1, w^2,\bar w^3)=\varphi_2(w^2)\qquad \mbox{ and }
\qquad\bet^3(\bar w^1, \bar w^2, w^3)=\varphi_3(w^3)\,.$$
Indeed, $\psi_2$ and $\psi_3$  are uniquely determined 
by $$\varphi_2(w^2)=e^{\bar w^1+\bar w^3}\,\psi_2(w^2)\qquad \mbox{ and }\qquad
\varphi_3(w^3)=e^{\bar w^1+\bar w^2}\,\psi_3(w^3).$$
This example shows how two arbitrary constants 
$\bar w^1+\bar w^2$ and $\bar w^1+\bar w^3$ may be ``hidden" in the arbitrary 
functions by replacing $\varphi_2$ with  $e^{-(\bar w^1+\bar w^3)}\varphi_2$ 
and $\psi_3$ with  $e^{-(\bar w^1+\bar w^3)}\varphi_3$.
 
The solution of the $\lam$-system is
\beq\label{ex-lam} 
	\lam^1=F\left(\frac{u^3\,u^2}{u^1}\right)
 	+\frac{u^3\,u^2}{u^1}F'\left(\frac{u^3\,u^2}{u^1}\right),
	\quad \lam^2=\lam^3=F\left(\frac{u^3\,u^2}{u^1}\right),
\eeq
where $F$ is an arbitrary function. In terms of the Riemann coordinates we have
$$ \lam^1(w)=h(w^1)+e^{w^1} h'(w^1),\quad \lam^2(w)=\lam^3(w)=h(w^1),$$
where $h$ is an arbitrary function. 
In accordance with Theorem 4.3 in \cite{jk1}, for a fixed point $(\bar w^1,\bar w^2,\bar w^3)\in \Omega$, 
an arbitrary function  $\phi(w^1)$, and a constant $\hat h$ there is a unique solution such that:
$$\lam^1(w^1,\bar w^2,\bar w^3)=\phi(w^1) \qquad\mbox{ and } 
\qquad\lam^2(\bar w^1,\bar w^2,\bar w^3)= \lam^3(\bar w^1,\bar w^2,\bar w^3)=\hat h.$$
Indeed, the solution \eq{ex-lam} is obtained by solving the ODE 
$\phi(w^1)=h(w^1)+e^{w^1} h'(w^1)$ for $h$ with initial value $\hat h=h(\bar w^1).$ 
\end{example}

\subsection{Examples with non-rich frames}
\subsubsection{\bf Non-rich frames admitting strictly hyperbolic conservative systems} 
Any frame corresponding to a non-rich Euler system (Example~\ref{gas_dyn}) is of this type. 
While non-rich Euler frames admit non-degenerate extensions, Examples~\ref{non-rich-rank1_(iib)} 
and \ref{non-rich-rank1_strict_hyp_case_3b} below show that this is not always the case.

\begin{example}\label{non-rich-rank1_(iib)}{\sc Non-rich frame  with strictly hyperbolic solutions 
of the $\lam$-system and exactly two vanishing $\bet^i$}: 
\[R_1=(-1,\,0,\, u^2+1)^T\,,\qquad R_2=\Big(\frac{u^3}{(u^2)^2-1},\,-1,\, u^1\Big)^T\,,
\qquad R_3=(1,\,0,\, 1-u^2)^T\,.\]
This is a non-rich frame with $\rank \eq{sev2}=\rank \eq{beta2}=1$. 
The unique algebraic $\lambda$-constraint 
is $2\lam^2=(1-u^2)\lam^1+(1+u^2)\lam^3$, and the $\lam$-system  
belongs to {\bf \em Case   IIa}  of Proposition~\ref{n=3lambda}. Its general  
solution  depends on two constants $C_1$, $C_2$ and is given by
\[\lam^1=C_1-2\,C_2,\qquad \lam^2=C_1+(u^2-1)\,C_2,\qquad\lam^3=C_1\,.\]
The flux in the corresponding conservative system \eq{cons} is given by 
\beq\label{bu_flux}
	f(u)=\left(\begin{array}{c}
	(C_1+C_2\,(u^2-1))\,u^1\,+C_2\,u^3,\\
	u^2\,(C_1-C_2+\frac 1 2 C_2\,u^2),\\
	C_2\,u^1\,(1-(u^2)^2)-C_2\,u^2\,u^3+(C_1-C_2)\,u^3
	\end{array}\right)\,.
\eeq
The unique algebraic $\beta$-constraint  is $(u^2-1)\beta^1=(u^2+1)\beta^3$. 
The general  solution of the $\bet$-system belongs to {\bf \em Case 2} in 
Theorem \ref{thm_non-rich_n=3_short} and is given by
$$\bet^1\equiv 0,\qquad \bet^2=F(u^2),\qquad\bet^3\equiv 0,$$
where $F$ is an arbitrary function.
The corresponding extensions (modulo affine parts) are given by
\[\eta(u^1,u^2,u^3) =G(u^2)\,,\qquad\quad\text{where $G''=F$.}\]
It is known that there are (uniformly) strictly hyperbolic systems of the 
type \eq{bu_flux} whose weak solutions may exhibit finite-time blowup 
in either BV or $L^\infty$; see Section 9.10 in \cite{daf}, and \cites{je,jy}. 
However, it not known whether this can occur for systems equipped with 
a strictly convex entropy. For further examples of blowup phenomena 
see \cites{sever, yo}. 
\end{example} 

\begin{example}\label{non-rich-rank1_strict_hyp_case_3b}
{\sc Non-rich frame  with strictly hyperbolic solutions 
of the $\lam$-system and exactly one vanishing $\bet^i$}: 
\[R_1=(u^1,\,u^2,\, u^3)^T\,,\qquad R_2=(u^2,\,u^1,\, u^3)^T\,,
\qquad R_3=(u^1,\,u^3,\, u^2)^T\,.\]
This is a non-rich frame with $\rank \eq{sev2}=\rank \eq{beta2}=1$. 
The unique algebraic $\lambda$-constraint is 
$(u^1-2\lam^2+u^3)\lam^1+(u^2-u^3)\lam^2+(u^2-u^1)\lam^3=0$, and the 
$\lam$-system belongs to {\bf \em Case   IIa}  of Proposition~\ref{n=3lambda}. 
Its general  solution  depends on two constants $C_1$, $C_2$ and is given by
\[\lam^1=C_1,\qquad \lam^2=C_1+\frac{C_2(u^1-u^2)}{(u^1+u^2+u^3)^2},
\qquad\lam^3=C_1+\frac{C_2(u^2-u^3)}{(u^1+u^2+u^3)^2}\,.\]

The unique algebraic $\beta$-constraint reduces to $\beta^1\equiv 0$.
In accordance with {\bf \em Case 3a} of Theorem \ref{thm_non-rich_n=3_short},
the general solution of the $\bet$-system depends on two arbitrary functions of 1 variable:
$$\bet^1\equiv 0,\qquad \bet^2=\frac{(u^2-u^1)^2}{u^1}F\Big(\frac{u^2+u^3}{u^1}\Big),
\qquad\bet^3=\frac{(u^3-u^2)^2}{u^1+u^2}G\Big(\frac{u^1+u^2}{u^3}\Big).$$
By exchanging $R_3$ in this example with the vector field $(u^1,\,u^3,\, u^3)^T$ 
we obtain a frame with the same type of $\lam$- and $\beta-$ dependencies. 
\end{example} 
\begin{remark}
The case of non-rich Euler frames together with the two examples above show 
that non-rich frames on $\Omega\subset\RR^3$, admitting strictly hyperbolic conservative systems \eq{cons},
may have none, one, or two associated $\bet^i$ vanishing. We have not been able
to find examples of this type for which all three $\bet^i$ vanish identically.
\end{remark}

\subsubsection{\bf{Non-rich frames with no corresponding strictly 
hyperbolic conservative systems}} 
There are many non-rich frames which do not admit strictly hyperbolic 
conservative systems, and some of these admit only trivial conservative systems. 
The size of the solution set of the corresponding $\bet$-system does not, in general, 
correspond to the size of the solution set of the $\lam$-system. 
In this category of frames there are examples where the $\bet$-system has only 
trivial solutions (Example~\ref{non-rich-rank1_(ia)}), non-trivial but degenerate 
solutions (Examples~\ref{non-rich-rank1_(iiib)}, \ref{ex_jk1_5.3}, \ref{ex_jk1_5.4},
\ref{deg-family}), as well as examples with non-degenerate 
solution (Examples~\ref{non-rich-rank1_(iid)}, \ref{non-rich-rank1_(ib)}, 
\ref{non-rich-rank1_(iid1)}).
\begin{example}\label{unequal_rank}{\sc $n=4$ frame with different algebraic 
ranks of the $\lam$- and the $\bet$-systems.} 
	This example shows that Proposition \ref{n=3_ranks} does not generalize to systems 
	with more than 3 equations: the algebraic parts of $\lam(\R)$ and $\bet(\R)$ may have 
	different ranks when $n\geq 4$. Consider the frame
	\[R_1=(1, 0, u^2, u^4)^T, \qquad R_2=(0, 1, u^1, 0)^T, \qquad 
	R_3=(u^3, 0, 1, 0)^T,  \qquad R_4=(1, 0, 0, 0)^T.\]
	In this case the rank of the algebraic part of the $\lam$-system is 3. On the other hand 
	the algebraic part of the $\bet$-system is given by 
	\[\bet^3=0\,,\qquad \bet^4=0\,,\qquad -\bet^3+u^1 \bet^4=0,\,\qquad \bet^3+(u^3)^2 \bet^4=0\,,\]
	which is of rank 2.
	The solution to the $\bet$ system depends on two arbitrary functions $F_1$ and $F_2$  of one variable:
	$$ \bet^1(u)=F_1(u^4)\,,\quad  \bet^2(u)=F_2(u^2)\,,\quad  \bet^3(u)=\bet^4(u)\equiv 0,$$
	while the $\lam$-system has only trivial solutions.
\end{example}

\begin{example}\label{non-rich-rank1_(ia)}{\sc a frame with non-trivial solutions of the 
$\lam$-system and only trivial solutions of the $\bet$-system:} 
	\[R_1=(u^1,\,-u^2,\, 0)^T,\qquad R_2=(-u^1,\, u^2,\, 1)^T, \qquad R_3=(1,\, 1,\, 1)^T.\]
	This is a non-rich frame with $\rank \eq{sev2}=1$ on the subset of $\RR^3$ where 
	$u_1\neq-u_2$. The $\lam$-system belongs to {\bf \em Case   IIb}  of 
	Proposition~\ref{n=3lambda} and has a non-trivial solution, depending on 1 constant
	and 1 function of 1 variable:
	\[\lam^1=\lam^2\equiv C,\quad \lam^3=(u^1+u^2)\,F(u^1\,u^2)+C.\]
	The only solution of the $\bet$-system is the trivial solution
	\[\bet^1=\bet^2=\bet^3\equiv 0\,.\]
	There are only trivial (affine) extensions in this
        case. This is an example of {\bf \em  Case 1} in Theorem \ref{thm_non-rich_n=3_short}.
\end{example}

\begin{example}\label{non-rich-rank1_(iiib)}{\sc a frame with non-trivial solutions of the 
$\lam$-system and non-trivial, but degenerate solutions of the $\bet$-system:}
	\[R_1=(1,\, u^2,\, 0)^T,\qquad R_2=(u^3,\, 1,\, 0)^T, \qquad R_3=(0,\, 0,\, 1)^T.\]
	This is a non-rich frame on $\RR^3$ with $\rank \eq{sev2}=1$. The $\lam$-system belongs 
	to {\bf \em Case   IIb}  of Proposition~\ref{n=3lambda} and its general solution depends on 1 constant
	and 1 function of 1 variable:
	\[\lam^1=\lam^2\equiv C,\quad \lam^3=H(u^3).\]
	 Two of the $\beta$'s vanish identically, while the 
	third depends on 1 function of 1 variable:
	\[\bet^1=\bet^2\equiv 0\,,\qquad \bet^3 =F(u^3)\,.\]
	The corresponding extensions (modulo affine parts) are
        given by
        \[\eta(u^1,u^2,u^3) =G(u^3)\,,
        \qquad\qquad\text{where $G''=F$.}\]
        	This is an example of {\bf \em  Case 2} in Theorem \ref{thm_non-rich_n=3_short}.
\end{example}
\begin{example}\label{ex_jk1_5.3}{\sc a frame with non-trivial solutions of the $\lam$-system 
and non-trivial, but degenerate solutions of the $\bet$-system:}
$$R_1=(1,\,0,\,0)^T,\qquad  R_2=(u^2,\,u^3,\, 1)^T, \qquad  R_3=(0,\,1,\,0)^T.$$
	This is a non-rich frame on $\RR^3$ with $\rank \eq{sev2} =\rank\eq{beta2}=1$. 
	The $\lam$-system belongs to {\bf \em Case IIb} of Proposition~\ref{n=3lambda} 
	and its general solution depends on 1 constant and 1 function of 1 variable:
	\[\lam^1=\lam^2\equiv C,\quad \lam^3=H\big((u^3)^2-2\,u^2\big).\]
	The general  solution of $\bet$-system  depends on 2 arbitrary functions of one variable:
	$$\bet^1=0,\qquad 
	\bet^2=\textstyle\frac{1}{2} F_1\big((u^3)^2-2 u^2\big)+F_2(u^3),\qquad  
	\bet^3= F_1'\big((u^3)^2-2\,u^2\big)\,.$$
	The corresponding extensions (modulo affine parts) are given by
        \[\eta(u^1,u^2,u^3) =\textstyle\frac{1}{4}G_1\big((u^3)^2-2\,u^2\big) + G_2(u^3)\,,
        \qquad\text{where $G_1'=F_1$ and $G_2''=F_2$.}\]
        This is an example of {\bf \em  Case 3a} in Theorem \ref{thm_non-rich_n=3_short}.
\end{example} 

\begin{example}\label{ex_jk1_5.4}   
{\sc two frames with only trivial solutions of the $\lam$-system and non-trivial, 
but degenerate solutions of the $\bet$-system:}
	$$R_1=(1,\,0,\, 0)^T,\qquad  R_2=(u^2,\,u^3,\, -u^2)^T,\qquad  R_3=(u^1+u^3,\,1,\, 0)^T.$$
	This is a non-rich frame on the subset of $\RR^3$, where $u^2\neq 0$ with 
	$\rank \eq{sev2}=1$. The $\lam$-system belongs to {\bf \em  Case  IIb} of 
	Proposition~\ref{n=3lambda}, and admits only trivial solutions:
	$$\lam^1=\lam^2=\lam^3\equiv C\in\RR.$$
	The general  solution of $\bet$-system  depends on 2 arbitrary functions of 
	one variable:
	$$\bet^1=0,\quad \bet^2=(u^2)^2\left[\int_*^{u^2}
	F_1(\tau^2+(u^3)^2)\Big[1+\frac{(u^3)^2}{\tau^2}\Big]\,d\tau+F_2(u^3)\right],
	\quad\bet^3=u^2\,F_1\big((u^2)^2+(u^3)^2\big).$$
	The corresponding extensions (modulo affine parts) are given by
        \[\eta(u^1,u^2,u^3) =\textstyle\frac{1}{2}\int_*^{u^2}G_1\big(\tau^2+(u^3)^2\big)\,d\tau + G_2(u^3)\,,
        \qquad\text{where $G_1'=F_1$ and $G_2''=F_2$.}\]
        This is an example of {\bf \em  Case 3a} in Theorem \ref{thm_non-rich_n=3_short}.
        
        Another example of the same is provided by 
        $$R_1:=(0,\, 0,\, 1)^T,\qquad  R_2=(0,\, 1,\, u^1)^T,\qquad  R_3=(u^3,\, 0,\, 1)^T.$$
        This frame belongs to {\bf \em Case IIa} of Proposition~\ref{n=3lambda} and to 
        {\bf \em  Case 3a} of Theorem \ref{thm_non-rich_n=3_short}. Its $\lambda$-system
        admits only trivial solutions $\lam^1=\lam^2=\lam^3\equiv C\in\RR$, while
	the general  solution of the $\bet$-system depends on 2 arbitrary functions of one variable:
	$$\bet^1=0,\qquad \bet^2 =F_2(u^2),\qquad  \bet^3 =F_1(u^1)\,(u^3)^2\,.$$
\end{example}

%

\begin{example}\label{deg-family}
{\sc a family of frames with non-trivial solutions of the $\lam$-system and non-trivial, 
but degenerate solutions of the $\bet$-system:}
\[R_1=(0,\,u^2,\, u^3)^T\,,\qquad R_2=(g(u^1),\,0,\, u^3)^T\,,
	\qquad R_3=(1,\,1,\, 0)^T\,.\]
These vector-fields form a frame on an open subset of $\RR^3$ where 
$u^3\,(u^2+g(u^1))\neq 0$. On this set the algebraic part \eq{beta2} of 
$\bet(\R)$ for this frame is
\beq\label{alg-deg-fanily}
	g'(u^1)\,\,\bet_1+\bet^2-g(u^1)\,u^2\,\bet^3 = 0\,.
\eeq
For a generic function $g(u^1)$, all three $\bet$'s are involved in the algebraic relation 
and we are in the situation described in Section~\ref{case(i)}.
The general solution of the $\bet$-system depends on 1 arbitrary constant $C$:
\[ \bet^1=0\, \quad \bet^2=C(g(u^1)+u^2), \quad \bet^3=C\,\frac{g(u^1)+u_2}{g(u^1)\, u^2}. \]
This is an Example of {\bf \em  Case 3b} in Theorem \ref{thm_non-rich_n=3_short}.

The number of $\bet$'s involved in the algebraic relations drops in the following specific 
cases:
\begin{enumerate}
\item[(i)] If $g(u^1)\equiv k\neq0\in\RR$ then \eq{alg-deg-fanily} reduces to 
$\bet^2-k\,u^2\,\bet^3=0$ and involves only two of $\bet^i$. We are then in the situation 
descried in Section~\ref{case(ii)}, but the general solution of
the $\bet$-system still depends on 1 arbitrary constant:
\[ \bet^1=0\, \quad \bet^2=C(k+u^2), \quad \bet^3=C\,\frac{k+u^2}{u^2}\,. \]
Thus, this particular case also falls into {\bf \em  Case 3b} of Theorem \ref{thm_non-rich_n=3_short}.
\item[(ii)] If $g\equiv 0$ then the algebraic $\beta$-relation \eq{alg-deg-fanily} reduces to $\bet^2=0$. 
This situation is described in Section~\ref{case(iii)}, and the general solution of
the $\bet$-system depends on 2 arbitrary functions of 1 variables:
\[ \bet^1=(u^2)^2\,G_1(u^2-u^1)\, \quad \bet^2=0, \quad \bet^3=G_2(u^1)\,. \]
This particular case falls into {\bf \em  Case 3a} of Theorem \ref{thm_non-rich_n=3_short}.
\end{enumerate}
Interestingly, the type of the  $\lam$-system does not depend on $g$.  
The algebraic part of $\lam(\R)$ is always equivalent to $\lam^1-\lam^2=0$, such  
that we are in {\bf \em  Case IIb} of Proposition~\ref{n=3lambda}. Its general solution 
depends on 1 arbitrary function of 1 variable and 1 constant:
\begin{align}
\mbox{if $g\neq 0$, then } & \lam^1=\lam^2\equiv K\,\mbox{ and } 
\lam^3=\frac{u^2+g(u^1)}{g(u^1)\,u^2}\, F\left(\frac{u^3}{u^2}
\,e^{-\int\frac{d\,u^1}{g(u^1)}}\right)+K;\\
\mbox{if $g\equiv 0$, then } & \lam^1=\lam^2\equiv K\,\mbox{ and }  \lam^3=F(u^1).
\end{align}
\end{example} 
\begin{example}\label{non-rich-rank1_(iid)}
	{\sc  a family of  frames with non-trivial solutions of the $\lam$-system and 
	non-degenerate solutions of the $\bet$-system:}
	\[R_1=(1,\, g(u),\, 0)^T,\qquad R_2=(-g(u),\, 1,\, 0)^T, \qquad R_3=(0,\, 1,\, 1)^T.\]
	The algebraic part  \eq{beta2} of $\bet(\R)$ for this frame is
	\[\frac{ \del_2 g+\del_3 g}{1+g^2}\,( \bet^2-\bet^1)=0.\]
	If $\del_2 g+\del_3 g=0 \Leftrightarrow g(u)=h(u_1,u_2-u_3)$ for some function 
	$h$ of two variables, then the frame is rich of rank 0, and the general solution of 
	the $\bet$ (as well as for the $\lam$ system) depends on 3 arbitrary functions of 
	1 variable. 
	
	Otherwise this is a non-rich frame on  $\RR^3$ with 
	$\rank \eq{sev2}=\rank \eq{beta2}=1$. The $\lam$-system belongs 
	to {\bf \em  Case IIb}  of Proposition~\ref{n=3lambda} and its general 
	solution depends on 1 constant and 1 function of 1 variable: 
	\[\lam^1=\lam^2\equiv C,\quad\lam ^3=F(u^3)\]
	The general solution of the $\bet$-system depends on 1 constant 
	and 1 function of 1 variable:
	\[\bet^1\equiv \bet^2=K\,(1+g(u)^2)\,,\qquad \bet^3 =F(u^3)\,.\]
	The corresponding extensions (modulo affine parts) are
        given by
        \[\eta(u^1,u^2,u^3) =\textstyle\frac{K}{2} \left[(u^1)^2 +(u^2- u^3)^2\right] +G(u^3)\,,
        \qquad\text{where $G''=F$.}\]
	 This is an example of {\bf \em  Case 4a} in Theorem \ref{thm_non-rich_n=3_short}.
\end{example}

\begin{example}\label{non-rich-rank1_(ib)}
{\sc a frame with non-trivial solutions of the $\lam$-system and non-degenerate 
solutions of the $\bet$-system:}
	\[R_1=(0,\, u^2,\, u^3)^T,\qquad R_2=(u^1,\, 0,\, u^3)^T, \qquad R_3=(1,\, 1,\, 0)^T.\]
	This is a non-rich frame  on the open subset of $\RR^3$ where $u^1+u^2\neq 0$ with 
	$\rank \eq{sev2}=\rank\eq{beta2}=1$. The $\lam$-system belongs to {\em \bf  Case IIb}  
	of Proposition~\ref{n=3lambda}, and its general solution depends on 1 constant
	and 1 function of 1 variable:
	\[\lam^1=\lam^2\equiv C,\quad\lam^3=\frac{u^1+u^2}{u^1\,u^2} F\left(\frac{u^3}{u^1\,u^2}\right)+C\]
	The general solution of the $\bet$-system depends on 2
        arbitrary constants:
	\[\bet^1= (K_1-K_2)(u^1+u^2)\,,\qquad
	\bet^2=K_2(u^1+u^2)\,,\qquad \bet^3
        =K_1\frac{(u^1+u^2)}{u^1u^2}\,.\]
        The corresponding extensions (modulo affine parts) are
        given by
        \[\eta(u^1,u^2,u^3) =K_1 \big[u^1\ln u^1 +u^2\ln u^2
        - u^3\ln u^3\big] +K_2(u^1-u^2)\ln u^3\,.\]
        This is an example of {\bf \em  Case 4b} in Theorem \ref{thm_non-rich_n=3_short}.
\end{example}

\begin{example}\label{non-rich-rank1_(iid1)} {\sc a frame with non-trivial solutions of the 
$\lam$-system and non-degenerate solutions of the $\bet$-system:}
\[R_1=(1,\,u^2,\, u^3)^T\,,\qquad R_2=(1,\,0,\, u^3)^T\,,\qquad R_3=(1,\,1,\, 0)^T\,.\]
	This is a non-rich frame on an open subset of $\RR^3$ where $u^2\neq 0$ with $\rank \eq{sev2}=1$. The $\lam$-system  
	belongs to {\em \bf  Case IIb}  of Proposition~\ref{n=3lambda} and its general solution  depends on 1 constant
	and 1 function of 1 variable:
	\[\lam^1=\lam^2\equiv C,\quad\lam ^3=F(u^3\,e^{-u^1})\]
	The $\bet$-system admits non-degenerate solutions and the general solution 
	depends on 1 constant:
	\[\bet^1= -K\, u^2\,,\qquad \bet^2=K\,
        u^2\,,\qquad \bet^3 \equiv K\,.\] 
        The corresponding extensions (modulo affine parts) are
        given by
        \[\eta(u^1,u^2,u^3) =K\, \left[\textstyle \frac{1}{2}(u^1)^2 +(1- u^2)\ln u^3\right]\,.\]
	This is an example of {\bf \em  Case 4c} in Theorem \ref{thm_non-rich_n=3_short}.
\end{example}



\end{document}